\documentclass{amsart}
\usepackage[T1]{fontenc}
\usepackage[utf8]{inputenc}

\usepackage{amssymb,color}
\usepackage{amsfonts}
\usepackage{amsmath}
\usepackage{euscript}
\usepackage{enumerate}
\usepackage{graphics}
\usepackage{graphicx}
\usepackage[all,cmtip]{xy}
\usepackage{tikz}
\usetikzlibrary{arrows}

\usepackage{hyperref}
\usepackage{pdfsync}
\newtheorem{theorem}{Theorem}[section]
\newtheorem{proposition}[theorem]{Proposition}
\newtheorem{corollary}[theorem]{Corollary}
\newtheorem{lemma}[theorem]{Lemma}

\theoremstyle{definition}

\newtheorem{remark}[theorem]{Remark}

\newtheorem{conj}[theorem]{Conjecture}
\numberwithin{equation}{section}

\begin{document}

\newcommand{\R}{\mathds{R}}
\newcommand{\Q}{\mathds{Q}}
\newcommand{\cok}{Coker}
\newcommand{\Rt}{\mbox{\newmatha R}}
\newcommand{\Rs}{\mbox{\newmathb R}}
\newcommand{\N}{\mathbb{N}}
\newcommand{\Z}{\mathbb{Z}}
\newcommand{\C}{\mathbb{C}}
\newcommand{\E}{\mathcal{E}}
\newcommand{\G}{\mathcal{G}}
\newcommand{\g}{\mathfrak{g}}
\newcommand{\s}{\mathfrak{sl}}
\newcommand{\A}{\mathcal{A}}
\newcommand{\F}{\mathbb{F}}
\renewcommand{\sl}{\mathfrak{sl}}
\newcommand{\gl}{\mathfrak{gl}}
\newcommand{\cad}{\ar@{}[dr]|{\circlearrowleft}}
\newcommand{\caad}{\ar@{}[ddr]|{\circlearrowleft}}
\newcommand{\cadd}{\ar@{}[drr]|{\circlearrowleft}}
\newcommand{\cai}{\ar@{}[dl]|{\circlearrowleft}}
\newcommand{\verteq}{\rotatebox{90}{$\,=$}}



\title[Tensor products and intertwining operators]{Tensor products and intertwining operators for
 uniserial representations of the Lie algebra
$\sl(2)\ltimes V(m)$}

\author{Leandro Cagliero}
\address{FaMAF-CIEM (CONICET), Universidad Nacional de C\'ordoba,
Medina Allende s/n, Ciudad Universitaria, 5000 C\'ordoba, Rep\'ublica
Argentina.}
\email{cagliero@famaf.unc.edu.ar}

\author{Iv\'an G\'omez Rivera}
\address{FaMAF-CIEM (CONICET), Universidad Nacional de C\'ordoba,
Medina Allende s/n, Ciudad Universitaria, 5000 C\'ordoba, Rep\'ublica Argentina.}
\email{ivan.gomez.rivera@mi.unc.edu.ar}
\thanks{This research was partially supported by an NSERC grant, CONICET
PIP 112-2013-01-00511, 
PIP 112-2012-01-00501, 
MinCyT C\'ordoba, 
FONCYT Pict2013 1391, 
SeCyT-UNC
33620180100983CB}



\begin{abstract} Let $\mathfrak{g}_m=\mathfrak{sl}(2)\ltimes V(m)$, $m\ge 1$, where $V(m)$ is the irreducible $\mathfrak{sl}(2)$-module of dimension $m+1$ viewed as an abelian Lie algebra. It is known that the isomorphism classes of uniserial $\mathfrak{g}_m$-modules consist of a family, say of type $Z$, containing modules of arbitrary composition length, and some exceptional modules with composition length $\le 4$.

  Let $V$ and $W$ be two uniserial $\mathfrak{g}_m$-modules of type $Z$. In this paper we obtain the $\mathfrak{sl}(2)$-module decomposition of $\text{soc}(V\otimes W)$ by giving explicitly the highest weight vectors. It turns out that $\text{soc}(V\otimes W)$ is multiplicity free. Roughly speaking, $\text{soc}(V\otimes W)=\text{soc}(V)\otimes \text{soc}(W)$ in half of the cases, and in these cases we obtain the full socle series of $V\otimes W$ by proving that $ \text{soc}^{t+1}(V\otimes W)=\sum_{i=0}^{t} \text{soc}^{i+1}(V)\otimes \text{soc}^{t+1-i}(W)$ for all $t\ge0$.
  
  As applications of these results, we obtain for which $V$ and $W$, the space of $\mathfrak{g}_m$-module homomorphisms $\text{Hom}_{\mathfrak{g}_m}(V,W)$ is not zero, in which case is 1-dimensional. Finally we prove, for $m\ne 2$, that if $U$ is the tensor product of two uniserial $\mathfrak{g}_m$-modules of type $Z$, then the factors are determined by $U$. We provide a procedure to identify the factors from $U$.
\end{abstract}

\maketitle

\section{Introduction}\label{sec:intro}
We fix throughout a field $\F$ of characteristic zero. All Lie algebras and representations considered in this paper are assumed to be finite dimensional over $\F$, unless explicitly stated otherwise.

It is generally acknowledged that the problem of classifying all indecomposable
finite dimensional representations of a Lie
algebra is intractable, one of the most clear manifestation of this
is given in \cite{GP} 
for abelian
Lie algebras of dimension greater than or equal to $2$. 
This is also discussed in \cite{Sa} 
for the 3-dimensional euclidean Lie
algebra $\mathfrak{e}(2)$, and in \cite{M} for virtually any complex Lie algebra other than semisimple
or 1-dimensional.

Rather than attempting to classify all indecomposable modules for a given
Lie algebra, or a family of Lie algebras, it would be very interesting to identify a class of representations that is
sufficiently limited so that we can have a reasonably comfortable handling of them and, at the same time, large enough to include many representations that appear naturally in problems of interest. 
Just to mention an example, 
let $A$ be a finite dimensional (associative or Lie) algebra and let  
$\text{Der}(A)$ be its Lie algebra of derivations. 
Except for very exceptional cases, $\text{Der}(A)$ 
is not semisimple.  
We know that $\text{Der}(A)$ 
acts in various objects associated to $A$, 
for instance in its (Hochschild or Lie) cohomology. 
There are many results in cohomology obtained by considering  
the action of a Levi factor of $\text{Der}(A)$ 
(and using the highest weight theory) disregarding the action
of its solvable radical. If we wanted to describe (and/or make use of) the 
whole $\text{Der}(A)$-module structure of the cohomology of $A$, 
there is no a standard way to do this 
due to the lack of knowledge we have of an appropriate class of  representations of $\text{Der}(A)$. Moreover, this could be specially 
useful if we wanted to describe the whole 
Gerstenhaber algebra $H\!H(A)$ of an 
associative algebra $A$.

Many authors have considered the idea of describing or classifying
a special class of representations of non-semisimple Lie algebras. 
For instance, A. Piard \cite{Pi1} analyzed thoroughly the indecomposable modules $U$,
of the complex Lie algebra $\mathfrak{sl}(2)\ltimes \mathbb{C}^2$,
 such that $U/\text{rad}(U)$ is irreducible. 
 More recently, various families of indecomposable modules over various types of non-semisimple Lie algebras have been constructed and/or 
 classified, see for instance \cite{Ca2, Ca1,CMS, CM, DP, Dd, DKR,J}.
  
On the other hand, we have been systematically studying uniserial representations of 
Lie algebras. 
In the articles \cite{CS_JofAlg, CS_canadian, CS_JofAlgApp, CS_Comm, CGS1, CGS2, CLS,Casati2017IndecomposableMO}
we and other authors have classified all uniserial representations for many different families of Lie algebras. It is worth mentioning that, 
in the theory of finite dimensional representations of associative algebras, the class of uniserial ones is quite relevant, a foundational result here
is due to T. Nakayama \cite{Na} (see also \cite{ASS} or\cite{ARS}) and it states that every finitely generated module over a serial ring is a direct sum of uniserial modules. For more information 
in the associative case we refer the reader mainly to \cite{ASS, ARS, Pu}, see also 
\cite{BH-Z, H-Z,NGB}.
We point out that, for Lie algebras, 
when $\g$ is 1-dimensional, any representation 
is a direct sum of uniserial ones. We do not know if the is
a class of Lie algebras, apart from semisimples, for which this remains true.

If we want to pursue farther the idea of identifying a class of
Lie algebras representations based on the uniserial ones, a natural step forward 
is to study morphisms between them and 
the tensor category that they generate. 
The main goal of this article is to start this project 
with the family of the Lie algebras
$\g_m=\sl(2)\ltimes V(m)$,  $m\ge 1$, 
 where $V(m)$ is the irreducible 
$\sl(2)$-module of dimension $m+1$ viewed as an abelian Lie algebra. 
The uniserial $\g_m$-modules where classified in \cite{CS_JofAlg}
and the isomorphism classes consist of a general family 
$Z(a,\ell)$ and its duals, and some exceptional modules
with composition length $\le 4$. This is described below 
in Theorem \ref{thm.CS_Classification}. 
In the family $Z(a,\ell)$, $a$ and $\ell$ are arbitrary 
non-negative integers, $\ell+1$ is the composition length 
of $Z(a,\ell)$ 
and the socle of $Z(a,\ell)$ is isomorphic to 
the irreducible $\sl(2)$-module of dimension $a+1$. 
We call the uniserial modules $Z(a,\ell)$ and $Z(a,\ell)^*$
\emph{uniserials of type} $Z$, they constitute the vast majority 
of uniserial $\g_m$-modules. 

\medskip

Our main results are the following.  
First,  given two uniserials $V$ and $W$ 
of type $Z$, we obtain in Theorem \ref{thm:main} 
the $\sl(2)$-module decomposition 
of $\text{soc}(V\otimes W)$ by  explicitly giving its
highest weight vectors.
By duality, since the $\sl(2)$-module decomposition 
of $V\otimes W$ follows from the 
Clebsch-Gordan formula, the 
$\sl(2)$-module structure of the radical 
$\text{rad}(V\otimes W)$ can be derived
(see for instance \cite[Chapter V]{ASS}).
It turns out that $\text{soc}(V\otimes W)$ is multiplicity free 
($V\otimes W$ is not at all multiplicity free except for $V$ and $W$ irreducible). In some sense, $\text{soc}(V\otimes W)=\text{soc}(V)\otimes \text{soc}(W)$ in half of the cases
(including when $V=Z(a,\ell)^*$ and $W=Z(b,\ell')^*$). 
It turns out that, in these cases (see Corollary 
\ref{coro:soc_completo}),
\[
\text{soc}^{t+1}(V\otimes W)=\sum_{i=0}^{t}
\text{soc}^{i+1}(V)\otimes 
\text{soc}^{t+1-i}(W)
\]
for all $t\ge0$. 
In other words, this formula holds for all $t$ if and only if it holds for $t=0$.
One of the main steps towards Theorem \ref{thm:main} is to deal with 
the cases $V$ and $W$ with composition lengths equal to 2.
This is done in Theorem \ref{thm:length_2} and
the proof of it required lengthy and precise computations
in which the Clebsch-Gordan coefficients 
(the 3-$j$ symbols) were a crucial tool.
In order to extend Theorem \ref{thm:length_2} to the exceptional
uniserials it is necessary to work out
these computations, but they become harder and so far
we could only arrive to Conjecture \ref{conj:length2}.

Next we study the intertwining operators between uniserials $V$ and $W$ 
of type $Z$. 
From the multiplicity free structure of $\text{soc}(V\otimes W)$
we obtain that $\text{Hom}_{\g_m}(V,W)$ is either zero or 1-dimensional
and we derive from Theorem \ref{thm:main} in 
which cases $\text{Hom}_{\g_m}(V,W)\ne 0$ (see 
Theorem \ref{thm:invariants} and Corollary \ref{thm:intertwining}). 
A well known result of this flavor is the Bernstein-Gelfand-Gelfand
classification of the intertwining operators among Verma modules and
their generalizations (see \cite{BGG}).

Finally we prove, for $m\ne 2$,  
that if $U$ is the tensor product of 
two uniserial modules of type $Z$, then the factors are determined by $U$ (see Theorem \ref{thm:isomorphism}). Moreover, we give explicitly a procedure to identify the factors. 
This question about the uniqueness of the factorization of tensor products  is frequently addressed in the literature for irreducible modules. 
It is well known that in general, 
the tensor product of two modules (even when they are irreducible)
do not determine the factors.
A very basic example would be the tensor product of two irreducible modules that is itself irreducible, 
and this may happen even if the underlying group or algebra is indecomposable and none of the factors is 1-dimensional 
(see for instance \cite{BK} or \cite{MOROTTI_JPAA,MOROTTI_Rep_Theory} and the references within them).
In contrast, a celebrated result of C. S. Rajan \cite{Ra}
states that a tensor product of an arbitrary number of 
irreducible, finite dimensional representations 
of a simple Lie algebra over a field of characteristic zero determines 
uniquely the factors. This is also true in other categories of 
modules, see \cite{Venkatesh2012UniqueFO} for a generalization of 
Rajan's result to a natural category of representations of symmetrizable Kac-Moody algebras, or \cite{Reif2021OnTP} for a unique 
factorization result for some special irreducible 
representations of Borcherds-Kac-Moody algebras.
We are not aware of results dealing with this problem within 
a much larger class of modules such as the class of uniserials. 
We think that Theorem \ref{thm:isomorphism} remains valid for $m=2$
and that our proof only requires a small adjustment that we did not find so far. 

We close this introduction with some open 
questions closely related to this paper 
that are of our interest. 
\begin{enumerate}[\hspace{3mm}]
\item[\tiny$\bullet$]
What is the $\sl(2)$-module structure of $\text{soc}(V\otimes W)$
when $V$ and $W$ are exceptional uniserial $\g_m$-modules?
As we mentioned, we think that the answer for modules of composition length 2 is 
given in Conjecture \ref{conj:length2}. 
The general case should follow without major difficulties from 
the result for the case of composition length 2.

\item[\tiny$\bullet$]
Is it true, for $m=2$, the statement of 
Theorem \ref{thm:isomorphism}?

\item[\tiny$\bullet$]
Is it possible to extend  
Theorem \ref{thm:isomorphism} to an arbitrary number of uniserial modules?

\item[\tiny$\bullet$]
Given two uniserials $\g_m$-modules $V$ and $W$, what are, up to isomorphism, the extensions of $V$ by $W$? 
Is it possible to obtain, for any $m$,  
results that are similar to those obtained by A. Piard \cite{Pi1} for $m=1$?

\item[\tiny$\bullet$]
For which uniserials  $\g_m$-modules $V$ and $W$ is $V\otimes W$ indecomposable?
\end{enumerate}

\section{Preliminaries}
\subsection{The Clebsch-Gordan coefficients}
\label{Subsec.Clebsch-Gordan}
Recall that $\F$ is a field of characteristic zero and that all Lie algebras and representations are assumed to be finite dimensional 
over $\F$, unless explicitly stated otherwise.
Let 
\begin{equation}\label{eq.basis_sl2}
e=\begin{pmatrix}
0 & 1 \\
0 & 0
\end{pmatrix},\qquad
h=\begin{pmatrix}
1 & 0 \\
0 & -1
\end{pmatrix},\qquad
f=\begin{pmatrix}
0 & 0 \\
1 & 0
\end{pmatrix}
\end{equation}
be the standard basis of $\sl(2)$.
Let $V(a)$ be the irreducible $\sl(2)$-module with highest weight $a\ge0$.
We fix a basis $\{v_0^a,\dots,v_a^a\}$ of $V(a)$ relative to which the basis $\{e,h,f\}$ acts as follows:
\begin{align*}
e\, v_k^{a}=&\sqrt{
\frac{a}{2}\left(\frac{a}{2}+1\right)-
\left(\frac{a}{2}-k+1\right)\left(\frac{a}{2}-k\right)
}
v_{k-1}^{a},\\[2mm]
h\, v_k^{a}=&(a-2k)v_k^{a},\\[2mm]
f\, v_k^{a}=&\sqrt{
\frac{a}{2}\left(\frac{a}{2}+1\right)-
\left(\frac{a}{2}-k-1\right)\left(\frac{a}{2}-k\right)
}
v_{k+1}^{a},
\end{align*} 
where $0\leq k\leq a$ and $v_{-1}^a=0=v_{a+1}^a$.
The basis $\{v_0^a,\dots,v_a^a\}$ has been chosen in a convenient way to
introduce below the Clebsch-Gordan coefficients.
Note that, if we denote by $(x)_a$ the matrix of 
$x\in\sl(2)$ relative to the basis $\{v_0^a,\dots,v_a^a\}$, then 
$\{(e)_1,(h)_1,(f)_1\}$
are as in \eqref{eq.basis_sl2},
and 
\begin{equation*}
(e)_2=\begin{pmatrix}
0 & \sqrt{2} & 0 \\
0 & 0 & \sqrt{2} \\
0 & 0 & 0
\end{pmatrix},\qquad
(h)_2=\begin{pmatrix}
2 & 0 & 0 \\
0 & 0 & 0 \\
0 & 0 & -2
\end{pmatrix},\qquad
(f)_2=\begin{pmatrix}
0 & 0 & 0 \\
\sqrt{2} & 0 &  0 \\
0 & \sqrt{2} & 0
\end{pmatrix}. 
\end{equation*}
This means 
that we may assume that 
$\{v_0^2,v_1^2,v_2^2\}=\{-e,\frac{\sqrt{2}}{2}h,f\}$.

We know that $V(a)\simeq V(a)^*$ as $\sl(2)$-modules. 
More precisely, if $\{(v_0^a)^*,\dots,(v_a^a)^*\}$ 
is the dual basis of $\{v_0^a,\dots,v_a^a\}$ then 
the map 
\begin{equation}\label{eq.dual}
\begin{split}
V(a) & \rightarrow V(a)^* \\
v_k^a & \mapsto (-1)^{a-k} (v_{a-k}^a)^* 
\end{split}
\end{equation}
gives an explicit $\sl(2)$-isomorphism.

It is well known that the tensor product decomposition
of two irreducible $\sl(2)$-modules $V(a)$ and $V(b)$
 is 
\begin{equation}\label{eq.tensor}
V(a)\otimes V(b)\simeq V(a+b)\oplus V(a+b-2) \oplus \cdots \oplus V(|a-b|).
\end{equation}
This is the well known Clebsch-Gordan formula.
The 
\emph{Clebsch-Gordan  coefficients}
$CG(j_{1},m_{1};j_{2},m_{2}\mid j_3,m_3)$ are defined below and they 
provide an explicit $\sl(2)$-embedding 
$V(c) \rightarrow V(a)\otimes V(b)$ 
which is the following
\begin{align*}
V(c) & \rightarrow V(a)\otimes V(b) \\
v_k^c & \mapsto v_k^{a,b,c}
\end{align*}
where, by definition, 
\begin{equation}\label{eq.Vc_en_tensor}
v_k^{a,b,c}=\sum_{i,j} 
CG(\tfrac{a}{2},\tfrac{a}{2}-i;\,\tfrac{b}{2},\tfrac{b}{2}-j
\,|\,\tfrac{c}{2},\tfrac{c}{2}-k)\,
 v_i^a\otimes v_j^b,
\end{equation}
where the sum runs over all $i,j$ such that 
$\tfrac{a}{2}-i+\tfrac{b}{2}-j=\tfrac{c}{2}-k$
(in fact we could let $i,j$ run freely since the Clebsch-Gordan
coefficient involved is zero if $\tfrac{a}{2}-i+\tfrac{b}{2}-j\ne\tfrac{c}{2}-k$).
Since 
\begin{equation}\label{eq.Hom}
\text{Hom}(V(b),V(a))\simeq V(b)^*\otimes V(a) 
\simeq V(a)\otimes V(b) 
\end{equation}
it follows from \eqref{eq.dual} and \eqref{eq.Vc_en_tensor} that 
the map $V(c) \rightarrow \text{Hom}(V(b),V(a))$
given by 
\begin{align}
v_k^c  & \mapsto 
\sum_{i,j} 
CG(\tfrac{a}{2},\tfrac{a}{2}-i;\,\tfrac{b}{2},\tfrac{b}{2}-j
\,|\,\tfrac{c}{2},\tfrac{c}{2}-k)\,
 v_i^a\otimes v_j^b, \notag \\
 & \mapsto 
\sum_{i,j} (-1)^{b-j}
CG(\tfrac{a}{2},\tfrac{a}{2}-i;\,\tfrac{b}{2},\tfrac{b}{2}-j
\,|\,\tfrac{c}{2},\tfrac{c}{2}-k)\,
 v_i^a\otimes (v_{b-j}^b)^*, \notag \\ \label{eq.embedding_Hom}
 & \mapsto 
\sum_{i,j} (-1)^{j}
CG(\tfrac{a}{2},\tfrac{a}{2}-i;\,\tfrac{b}{2},-\tfrac{b}{2}+j
\,|\,\tfrac{c}{2},\tfrac{c}{2}-k)\,\,
(v_{j}^b)^* \otimes v_i^a
\end{align}
is an $\sl(2)$-module homomorphism.

We now recall briefly the basic definitions and 
facts about the Clebsch-Gordan coefficients.
We will mainly follow \cite{VMK}.

Given three non-negative integers or half-integers  $j_1,j_2,j_3$, we say that they \emph{satisfy
the triangle condition} if
$j_1+j_2+j_3$ is an integer and 
they can be the side lengths of a (possibly degenerate) 
triangle (that is
$|j_1-j_2|\le j_3\le j_1+j_2$).
We now define (see \cite[\S8.2, eq.(1)]{VMK})
\[
 \Delta(j_1,j_2,j_3)=\sqrt{\frac{(j_1+j_2-j_3)!(j_1-j_2+j_3)!(-j_1+j_2+j_3)!}{(j_1+j_2+j_3+1)!}}
 \]
if $j_1, j_2, j_3$ satisfies the triangle condition;
otherwise, we set $\Delta(j_1,j_2,j_3)=0$.

If in addition $m_1$, $m_2$ and $m_3$ are three integers or half-integers then 
the corresponding \emph{Clebsch-Gordan coefficient}
\[ 
CG(j_{1},m_{1};j_{2},m_{2}| j_3,m_3)
\]
is zero unless  $m_1+m_2= m_3$ and $|m_i|\le j_i$ for $i=1,2,3$. In this case, the following formula is valid for $m_3\ge 0$ and $j_1\ge j_2$ (see \cite[\S8.2, eq.(3)]{VMK})
 \begin{multline*}
 CG(j_{1},m_{1};j_{2},m_{2}\mid j_3,m_3)=
 \Delta(j_1,j_2,j_3)\,\sqrt{(2j_3+1) } \\[1mm]
 \times \sqrt{(j_1+m_1)!(j_1-m_1)!(j_2+m_2)!(j_2-m_2)!(j_3+m_3)!(j_3-m_3)! } \\[1mm]
 \times
 \sum_r\frac{(-1)^r}{r!(j_1\!+\!j_2\!-\!j_3\!-\!r)!(j_1\!-\!m_1\!-\!r)!(j_2\!+\!m_2\!-\!r)!(j_3\!-\!j_2\!+\!m_1\!+\!r)!(j_3\!-\!j_1\!-\!m_2\!+\!r)!},
 \end{multline*}
 where the sum runs through all integers
$r$ for which the argument of every factorial is non-negative.
If either $m_3< 0$ or $j_1< j_2$ 
we have
\begin{align}
 CG(j_{1},m_{1};j_{2},m_{2}\mid j_3,m_3)
 & = (-1)^{j_1+j_2-j_3}\;
 CG(j_{1},-m_{1};j_{2},-m_{2}\mid j_3,-m_3)\notag
  \\[2mm]\label{eq:swap}
 & = (-1)^{j_1+j_2-j_3}\;
 CG(j_{2},m_{2};j_{1},m_{1}\mid j_3,m_3). 
\end{align}
In addition, it also holds
\begin{equation}\label{eq.simmetryCG}
 CG(j_{1},m_{1};j_{2},m_{2}\mid j_3,m_3)
=(-1)^{j_1-m_1}
\sqrt{\frac{2j_3+1}{2j_2+1}}\;
 CG(j_{1},m_{1};j_{3},-m_{3}\mid j_2,-m_2).
\end{equation}
In the following sections we will need the following  particular values of the Clebsch-Gordan coefficients.
Here $a,b$ are integers and $i=0,\dots,a$, 
$j=0,\dots,b$. 

\begin{equation}\label{eq.extremoCG0}
CG(\tfrac{a}{2},\tfrac{a}{2}-i;\,\tfrac{b}{2},\tfrac{b}{2}-j
\,|\,\tfrac{a+b}{2},\tfrac{a+b}{2}-i-j)
=
\sqrt{\frac{a!b!(a+b-i-j)!(i+j)!}{i!j!(a+b)!(a-i)!(b-j)!}},
\end{equation}
\begin{multline}\label{eq.extremoCG1}
CG(\tfrac{a}{2},\tfrac{a}{2}-i;\,\tfrac{b}{2},j-\tfrac{b}{2}
\,|\,\tfrac{a-b}{2},\tfrac{a-b}{2}-i+j) \\
=(-1)^j
\sqrt{\frac{(a-i)!\;i!\;b!\;(a-b+1)!}
{(a+1)!\;j!\;(b-j)!\;(a-b-i+j)!\;(i-j)!}},
\end{multline}
\begin{multline}\label{eq.extremoCG11}
CG(\tfrac{a}{2},i-\tfrac{a}{2};\,\tfrac{b}{2},\tfrac{b}{2}-j
\,|\,\tfrac{b-a}{2},\tfrac{b-a}{2}+i-j) \\
=(-1)^{a}CG(\tfrac{b}{2},\tfrac{b}{2}-j;\,
\tfrac{a}{2},i-\tfrac{a}{2}
\,|\,\tfrac{b-a}{2},\tfrac{b-a}{2}+i-j) \\
=(-1)^j
\sqrt{\frac{(b-j)!\;j!\;a!\;(b-a+1)!}
{(b+1)!\;i!\;(a-i)!\;(b-a-j+i)!\;(j-i)!}},
\end{multline}
\begin{multline}\label{eq.extremoCG}
CG(
\tfrac{a}{2},\tfrac{a}{2}-i;\,
\tfrac{b}{2},\tfrac{b}{2}-j\,|\,
\tfrac{a+b}{2}-i-j,\tfrac{a+b}{2}-i-j) \\
=
(-1)^i
\sqrt{\frac{(a+b-2i-2j+1)!\;(i+j)!\;(a-i)!\;(b-j)!}
{(a+b-i-j+1)!\;(a-i-j)!\;(b-i-j)!\;i!\;j!}}.
\end{multline}

\subsection{Uniserial representations}

Given a Lie algebra $\g$ and a $\g$-module $V$, we say that $V$ is \emph{uniserial}  if it admits a unique composition series. 
In other words, $V$ is uniserial  if the socle series
\[
0 = \text{soc}^0(V)\subset \text{soc}^1(V) \subset \cdots \subset \text{soc}^n(V) = V
\]
is a composition series of $V$, that is, the socle factors $\text{soc}^{i}(V)/\text{soc}^{i-1}(V)$ are irreducible for all $1\le i \le n$. 
Recall that $\text{soc}^1(V)=\text{soc}(V)$ 
is the sum of all irreducible   
$\g$-submodules of $V$ and 
$\text{soc}^{i}(V)/\text{soc}^{i-1}(V)
=\text{soc}(V/\text{soc}^{i-1}(V))$.
Note that for uniserial modules, 
the \emph{composition length} $n$  of $V$ coincides with
 its socle length.

If the
 Levi decomposition of $\g$ is $\g = \mathfrak{s} \ltimes \mathfrak{r}$, (with $\mathfrak{r}$ the solvable radical and $\mathfrak{s}$ semisimple) we may choose irreducible 
 $\mathfrak{s}$-submodules $V_i\subset V$, $1\le i \le n$, such that 
 \begin{equation}\label{eq.soc_decomp}
 V=V_1\oplus \cdots \oplus V_n 
 \end{equation}
with 
$V_i
 \simeq \text{soc}^{i}(V)/\text{soc}^{i-1}(V)$ 
 as   $\mathfrak{s}$-modules and 
 \[
 \mathfrak{r} V_i \subset V_1\oplus \cdots \oplus V_i. 
 \]
In fact, if 
$[\mathfrak{s},\mathfrak{r}]=\mathfrak{r}$, then $\mathfrak{r} V_i \subset V_1\oplus \cdots \oplus V_{i-1}$, 
see Lemma \ref{lemma:soc} below. 
We say that \eqref{eq.soc_decomp} is the 
\emph{socle decomposition} of $V$, note that the order of the summands is relevant.

\begin{lemma}\label{lemma:soc}
If $\mathfrak{r}=[\mathfrak{s},\mathfrak{r}]$, then
$\text{soc}(U)=U^\mathfrak{r}$ for any $\g$-module $U$.
\end{lemma}
\begin{proof}
On the one hand, $U^\mathfrak{r}$ is a completely reducible 
$\mathfrak{s}$-submodule of $U$ and hence, 
a completely reducible $\g$-submodule, thus  
$U^\mathfrak{r}\subset \text{soc}(U)$.
On the other hand, if $U_1$ is an irreducible $\g$-submodule of $U$, 
since the characteristic of the field $\F$ is 0, 
we know $\mathfrak{r}=[\g,\g]\cap\mathfrak{r}$ acts trivially on $U_1$ 
(see \cite[Chapitre 1, \S5.3]{Bo}).
Hence $U_1\subset U^\mathfrak{r}$ and therefore 
$\text{soc}(U)\subset U^\mathfrak{r}$. 
\end{proof}

\begin{lemma}\label{lemma:soc_series}
Assume that  $\mathfrak{r}=[\mathfrak{s},\mathfrak{r}]$ and 
let $V$ be a $\g$-module such that it has a vector space
 decomposition $V=V_1\oplus \cdots \oplus V_n$ such that 
$\mathfrak{r} V_k \subset V_{k-1}$ for all $k=2,\dots,n$.
If $\text{soc}(V)=V_1$ then 
$\text{soc}^k(V)=V_1\oplus \cdots \oplus V_k$
for all $k=1,\dots,n$.
\end{lemma}
\begin{proof}
We proceed by induction on $k$. 
Assume the statement true 
for all $k$ with $1\le k< k_0 <n$ and let us prove 
that 
$\text{soc}^{k_0}(V)=V_1\oplus \cdots \oplus V_{k_0}$. 

Let $U=V/\text{soc}^{k_0-1}(V)$ and let 
$p:V\to U$ the corresponding projection. 
We point out that, since
\[
\text{soc}^{k_0-1}(V)=V_1\oplus \cdots \oplus V_{k_0-1}
\] 
is a 
$\g$-submodule, it follows that $p$ is a $\g$-module homomorphism
and 
\begin{equation}\label{eq:dec_U}
U=p(V_{k_0})\oplus \cdots \oplus p(V_{n})
\end{equation}
as vector spaces. 

We know, by the definition of the socle series, that
$\text{soc}^{k_0}(V)$ is the $\g$-submodule of $V$ satisfying 
\[
\text{soc}^{k_0}(V)/\text{soc}^{k_0-1}(V)=\text{soc}(U)
\] 
and it follows from Lemma \ref{lemma:soc} that 
\[
\text{soc}(U)=
\big(p(V_{k_0})\oplus \cdots \oplus p(V_{n})\big)^{\mathfrak{r}}.
\]
Let  $v=v_{k_0}+v_{k_0+1}+\hdots+v_n$ with 
$v_k\in V_{k}$, 
 $k=k_0,\dots,n$, such that $Xp(v)=0$
  for all $X\in\mathfrak{r}$.
  Since $p$ is a $\g$-module homomorphism, we have 
\[
p(Xv_{k_0})+p(Xv_{k_0+1})+\hdots+p(Xv_{n})
=0.
\]
The hypothesis $\mathfrak{r} V_k \subset V_{k-1}$  
implies $X(p(v_{k}))  = p(Xv_{k})\in p(V_{k-1})$ 
and hence $p(Xv_{k_0})=0$ and 
\[
p(Xv_{k_0+1})+\hdots+p(Xv_{n})
=0.
\]
Since $p|_{V_i}$ is injective for all $i\ge k_0$, 
it follows from \eqref{eq:dec_U} that
\[
Xv_i=0,\quad \text{for all $i\ge k_0+1$ and all $X\in\mathfrak{r}$.}
\]
Finally, it follows from 
Lemma \ref{lemma:soc} and the 
hypothesis $\text{soc}(V)=V_1$,
that $V^{\mathfrak{r}}=V_1$.
This implies that $v_i=0$ for all $i\ge k_0+1$.
Therefore, we have proved that 
$\big(
p(V_{k_0})\oplus \cdots \oplus p(V_{n})\big)^{\mathfrak{r}}=p(V_{k_0})$
and thus 
$\text{soc}(U)=p(V_{k_0})$.
This shows that 
\[
\text{soc}^{k_0}(V)=V_1\oplus \cdots \oplus V_{k_0}
\]
and the induction step is complete.
\end{proof}

\subsection{Uniserial representations of \texorpdfstring{$\sl(2)\ltimes V(m)$}{sl(2)xV(m)}.}\label{wdos}

In \cite{CS_JofAlg} it is obtained 
the classification, up to isomorphism,
of all the uniserial representations of the Lie algebra $\sl(2)\ltimes V(m)$, $m\ge1$,
when the underlying field is $\C$. Nevertheless, the classification remains true over any field $\F$ of characteristic 0. 
The main ingredients of this classification are 
the modules $E(a,b)$ and $Z(a,\ell)$ that we present below.

From now on, we fix $m\ge 1$ and set
$\g_m=\sl(2)\ltimes V(m)$. 
We have 
\[
\mathfrak{s} =\sl(2)\quad\text{ and }\quad \mathfrak{r} =[\mathfrak{s}, \mathfrak{r}]=V(m).
\]
It will be useful to have a special notation for the 
basis $\{v_0^m,\dots,v_m^m\}$ of $\mathfrak{r}=V(m)$ as part of the Lie algebra $\g_m$. Thus, we will denote the basis of $\mathfrak{r}$ by 
$\{e_0,\dots,e_m\}$.

If $a$ and $b$ are non-negative integers such that 
$\frac{m}2,\frac{a}2,\frac{b}2$ satisfy the triangle condition, 
it follows from \eqref{eq.tensor} and \eqref{eq.Hom}
that, up to scalar, there is a unique
$\sl(2)$-module homomorphism 
\[
\mathfrak{r}=V(m)\to \text{Hom}(V(b),V(a)).
\] 
This produces an action of 
$\mathfrak{r}$ on $V(a)\oplus V(b)$ given by 
$\mathfrak{r} V(a)=0$ and 
\begin{equation}\label{eq.actionV(m)}
e_s\, v_j^{b}=(-1)^j\,\sum_{i=0}^a
 CG(\tfrac{a}{2},\tfrac{a}{2}-i;\,\tfrac{b}{2},-\tfrac{b}{2}+j
\,|\,\tfrac{m}{2},\tfrac{m}{2}-s)\,
v_{i}^{a},\qquad s=0,\dots,m.
\end{equation}
Note that this is the same as \eqref{eq.embedding_Hom}.
Note also that the above sum has, in fact, at most  one summand, that is
\begin{equation}\label{eq.actionV(m)Bis}
e_s\, v_j^{b}=
\begin{cases}
0,& \text{if $i\ne j+s+\frac{a-b-m}{2}$;} \\[2mm]
(-1)^j CG(\tfrac{a}{2},\tfrac{a}{2}-i;\,\tfrac{b}{2},-\tfrac{b}{2}+j
\,|\,\tfrac{m}{2},\tfrac{m}{2}-s)\,\,v_{i}^{a},& 
\text{if $i=j+s+\frac{a-b-m}{2}$.}
\end{cases}
\end{equation}

This action, combined with the action of $\sl(2)$  defines a uniserial 
$\g_m$-module structure with composition length 2 on 
\[
E(a,b)=V(a)\oplus V(b).
\]
It is straightforward to see that 
$E(a,b)^*\simeq E(b,a)$.
The action given in \eqref{eq.actionV(m)}
is the main building block for all other uniserial 
$\g_m$-modules as follows. 

The above construction can be extended to arbitrary composition length 
\[
V(a_0)\oplus V(a_1)\oplus\cdots\oplus V(a_\ell)
\]
only when the sequence $a_i$ is monotone and $m=|a_i-a_{i-1}|$,  for all $i=1,\dots,\ell$.
More precisely, for $\alpha$ and $\ell$ non-negative integers, let 
$Z(\alpha,\ell)$ be the uniserial $\g_m$-module defined by
 \begin{equation}\label{eq.soc_decomp_Z}
Z(\alpha,\ell)=V(\alpha)\oplus V(\alpha+m)\oplus \cdots \oplus V(\alpha+\ell m)
 \end{equation}
as $\sl(2)$-module with
action of $\mathfrak{r}$
sending
\[
0\longleftarrow
V(\alpha)\longleftarrow 
V(\alpha+2m)\longleftarrow \dots\longleftarrow 
V(\alpha+\ell m)
\]
as indicated in \eqref{eq.actionV(m)}
(with $a=\alpha+(i-1)m$, $b=\alpha+im$, 
for $i=1,\dots,\ell$).

We notice that $Z(\alpha,0)=V(\alpha)$ ($\mathfrak{r}$ acts trivially) and 
$Z(\alpha,1)=E(\alpha,\alpha+m)$. 
The uniserial modules $Z(\alpha,\ell)$ 
and their duals will be called \emph{of type $Z$}, and they are the unique  
isomorphism classes of uniserial $\g_m$-modules of 
composition length $\ell+1$ for $\ell \ge 4$. 

For composition lengths $3$ and $4$, very few other
ways to ``combine'' the modules $E(a,b)$ are possible.
For composition length equal to $3$, given $0\le c\le 2m$ and 
$c\equiv 2m\mod 4$, 
 let 
\[
E_3(c)=
V(0)\oplus V(m)\oplus V(c)
\]
as $\sl(2)$-modules with action of 
$\mathfrak{r}$ sending 
\begin{center}
\begin{tikzpicture}[->,>=stealth',auto,node distance=2cm,thick]
  \node (0) {$0$};
  \node (1) [right of=0] {$V(0)$};
  \node (2) [right of=1] {$V(m)$};
  \node (3) [right of=2] {$V(c)$};

  \path[every node/.style={font=\sffamily\small}]
    (1) edge node [right] {} (0)
    (2) edge node [right] {} (1)
    (3) edge node [right] {} (2);
\end{tikzpicture}
\end{center}
  with the maps $V(c)\to V(m)$ and $V(m)\to V(0)$
given by \eqref{eq.actionV(m)}.

For composition length equal to $4$, if $m\equiv 0\mod 4$, 
there is a family of $\g_m$-modules, parameterized
by a non-zero scalar $t\in\F$, with a fixed socle decomposition. This is defined by 
\[
E_4(t)=
V(0)\oplus V(m)\oplus V(m)\oplus V(0)
\]
as $\sl(2)$-modules with action of 
$\mathfrak{r}$ sending the $\sl(2)$-modules as shown by the arrows
\begin{center}
\begin{tikzpicture}[->,>=stealth',auto,node distance=2cm,thick]
  \node (0) {$0$};
  \node (1) [right of=0] {$V(0)$};
  \node (2) [right of=1] {$V(m)$};
  \node (3) [right of=2] {$V(m)$};
  \node (4) [right of=3] {$V(0)$};

  \path[every node/.style={font=\sffamily\small}]
    (1) edge node [right] {} (0)
    (2) edge node [right] {} (1)
    (3) edge node [right] {} (2)
    (4) edge node [right] {} (3)
    (4) edge[bend right] node [left] {} (2);
\end{tikzpicture}
\end{center}
where the horizontal arrows are 
 given by \eqref{eq.actionV(m)} 
 and the bent arrow is $t$ times \eqref{eq.actionV(m)}.
We can now state one of the main results of \cite[Thm 10.1]{CS_JofAlg}.

\begin{theorem}\label{thm.CS_Classification}
The following list describes all the isomorphism classes of uniserial representations 
of $\g_m=\sl(2)\ltimes V(m)$.

\medskip

\noindent
\begin{tabular}{ll} 
Length 1. & $Z(a,0)=V(a)$, $a\ge0$. \\[2mm]
Length 2. & $E(a,b)$, with $a+b\equiv m\mod 2$ and 
 $0\le |a-b|\leq m\leq a+b$.  \\[2mm]
Length 3. & $Z(a,2)$, $Z(a,2)^*$, $a\ge0$; and \\[1mm]
      & $E_3(c)$ with $c\equiv 2m \mod 4$ and $0\le c\leq 2m$. \\[2mm]
Length 4. & $Z(a,3)$, $Z(a,3)^*$, $a\ge0$; and  \\[1mm]
      & $E_4(t)$, with $0\ne t\in\F$ (this exists only if $m\equiv 0\mod 4$). \\[2mm]
Length $\ell\geq 5$. & $Z(a,\ell-1)$, $Z(a,\ell-1)^*$, $a\ge0$.   \\[2mm]
\end{tabular}

\end{theorem}

\section{The socle of the tensor product of two uniserial $\g_m$-modules}

\subsection{General considerations}\label{sec:Gen_cons}

Given two $\g_m$-modules $V$ and $W$, it is clear that 
$\text{soc}(V)\otimes\, \text{soc}(W)\subset \text{soc}(V\otimes W)$.
Therefore, if $V$ and $W$ are uniserial $\g_m$-modules 
with socle decomposition (see \eqref{eq.soc_decomp} and the comments below it)
\begin{align*}
V &=
V(a_0)\oplus V(a_1)\oplus \hdots \oplus V(a_\ell), \\
W &=
V(b_0)\oplus V(b_1)\oplus \hdots \oplus V(b_{\ell'}) 
\end{align*}
($V(a_0)=\text{soc}(V)$, $V(b_0)=\text{soc}(W)$),
 we have
\[
V(a_0)\otimes V(b_0)\subset \text{soc}(V\otimes W).
\]

For convenience we assume $V(a_{i})=V(b_{j})=0$ for $i,j<0$.
We know from Theorem \ref{thm.CS_Classification}
that 
\[
\mathfrak{r}V(a_i)\subset V(a_{i-1})\oplus V(a_{i-2})\quad
\text{and}
\quad 
\mathfrak{r}V(b_j)\subset V(b_{j-1})\oplus V(b_{j-2})
\]
for all $i\le \ell$, $j\le\ell'$ (in fact
we know that 
 $\mathfrak{r}V(a_i)\subset V(a_{i-1})$
and $\mathfrak{r}V(b_j)\subset V(b_{j-1})$ 
except for cases of uniserials of type $E_4$).
This implies that 
\begin{equation}\label{eq:filtration}
\mathfrak{r}\;\Big(
\bigoplus_{i+j=t} 
V(a_i)\otimes V(b_j)
\Big)\subset 
\bigoplus_{i+j<t} 
V(a_i)\otimes V(b_j).
\end{equation}
Moreover, we point out for future use that 
if neither $V$ nor $W$ is of type $E_4$ then 
\begin{equation}\label{eq:r in type Z}
\mathfrak{r}\,\big(V(a_i)\otimes V(b_{j})\big)
\subset V(a_{i-1})\otimes V(b_{j})\;\oplus\; 
V(a_i)\otimes V(b_{j-1})
\end{equation}

Since $\mathfrak{r}=[\mathfrak{s},\mathfrak{r}]$,
it follows from Lemma \ref{lemma:soc} that
$\text{soc}(U)=U^\mathfrak{r}$ for any $\g_m$-module $U$.
Hence, it follows from \eqref{eq:filtration} that
\begin{align}
\text{soc}(V\otimes W)
&=
\bigoplus_{t=0}^{\ell+\ell'}
\Big(
\text{soc}(V\otimes W) \cap \bigoplus_{i+j=t} V(a_{i})\otimes V(b_j)\Big) \notag \\\label{eq:soc=r-invariants}
&=
\bigoplus_{t=0}^{\ell+\ell'}
\Big(\bigoplus_{i+j=t} V(a_{i})\otimes V(b_j)\Big)^{\mathfrak{r}}.
\end{align}
For $t=0,\dots,\ell+\ell'$, let us define 
\begin{equation*}
S_t=S_t(V,W)=
\Big(\bigoplus_{i+j=t} V(a_{i})\otimes V(b_j)\Big)^{\mathfrak{r}}
\end{equation*}
so that $
\displaystyle{\text{soc}(V\otimes W)=\bigoplus_{t=0}^{\ell+\ell'}S_t}$.
It is clear that 
\begin{equation*}
S_0=\Big(V(a_{0})\otimes V(b_0)\Big)^{\mathfrak{r}}=
V(a_{0})\otimes V(b_0)=
\text{soc}(V)\otimes\text{soc}(W)
\end{equation*}
and hence 
\begin{equation}\label{eq:soc=socxsoc+S_0}
\text{soc}(V\otimes W) =
\text{soc}(V)\otimes\text{soc}(W)\;\oplus\; \bigoplus_{t=1}^{\ell+\ell'}S_t
\end{equation}
as  $\sl(2)$-modules. 
Hence, in order to obtain 
the $\sl(2)$-decomposition of 
$\text{soc}(V\otimes W)$ we need to find 
the highest weight vectors in $S_t$ (specially for $t\ge1$)
that are 
annihilated by $\mathfrak{r}$.

Given $v\in V(a_i)\otimes V(b_j)$ and $e_s\in \mathfrak{r}$, let 
\begin{equation}\label{eq.e_sv}
e_sv=(e_sv)_1 +(e_sv)_2 + (e_sv)_3
\end{equation}
where 
\begin{align*}
(e_sv)_1&\in V(a_{i-1}) \otimes V(b_j), \\[1mm]
(e_sv)_2&\in V(a_i)\otimes V(b_{j-1}), \\[1mm]
(e_sv)_3&\in V(a_{i-2}) \otimes V(b_j) \oplus V(a_i)\otimes V(b_{j-2})
\end{align*} 
(note that $(e_sv)_3=0$ if neither $V$ nor $W$ is of type $E_4$). 
It is clear that $(e_sv)_1=0$ if $i=0$ and $(e_sv)_2=0$ if $j=0$, the following lemma states a sort of  converse of this for highest weight vectors in 
$V(a_i)\otimes V(b_j)$.

\begin{lemma}\label{lemma.no nulo arriba e izquierda}
Let 
$V=V(a_0)\oplus \hdots \oplus V(a_\ell)$ and 
$W=V(b_0)\oplus \hdots \oplus V(b_{\ell'})$,
with $\ell,\ell'\ge1$, be the socle decomposition of 
two uniserial $\g_m$-modules. 
If $v_0\in V(a_{i_0})\otimes V(b_{j_0})$ is a highest weight vector then:
\begin{enumerate}[(i)]
\item $(e_sv_0)_1=0$ for all $s=0,\dots, m$ if and only if $i_0=0$.
\item $(e_sv_0)_2=0$ for all $s=0,\dots, m$ if and only if $j_0=0$.
\end{enumerate}
\end{lemma}

\begin{proof}
By symmetry it suffices to prove (i). 
As we already mentioned, it is clear the ``if'' part and 
thus we will prove the ``only if'' part.
So let us assume $i_0>0$ and let us prove 
$(e_sv_0)_1\ne 0$ for some $s=0,\dots, m$.

If $c$ is the weight of $v_0$, we may assume 
that (see \eqref{eq.Vc_en_tensor} and \eqref{eq.extremoCG})
\begin{align*}
v_0
=v_0^{a,b,c}
&=\sum_{i+j=\frac{a+b-c}{2}} 
CG(\tfrac{a}{2},\tfrac{a}{2}-i;\,
\tfrac{b}{2},\tfrac{b}{2}-j
\,|\,\tfrac{c}{2},\tfrac{c}{2})\,
 v_i^{a}\otimes v_j^{b} \\
&=\sum_{i+j=\frac{a+b-c}{2}} 
(-1)^i
\sqrt{\frac{(a+b-2i-2j+1)!(i+j)!(a-i)!(b-j)!}
{(a+b-i-j+1)!(a-i-j)!(b-i-j)!\,i!\,j!}}\,
 v_i^{a}\otimes v_j^{b} \\
&=
\sqrt{\frac{(c+1)!\;a!}
{(\frac{a+b+c}{2}+1)!\;(\frac{a-b+c}{2})!}}\,\,
 v_0^{a}\otimes v_{\frac{a+b-c}{2}}^{b} 
 \quad + \quad
 \sum_{j<\frac{a+b-c}{2}} q_j\;
 v_{\frac{a+b-c}{2}-j}^{a}\otimes v_j^{b} 
\end{align*}
where $a=a_{i_0}$ and $b=b_{j_0}$
and $q_j\in\F$ is some scalar.
Now, for $s=0,\dots, m$, we have
\begin{multline*}
(e_sv_0)_1=\sqrt{\frac{(c+1)!\;(a_{i_0})!}
{(\frac{a_{i_0}+b_{j_0}+c}{2}+1)!\;(\frac{a_{i_0}-b_{j_0}+c}{2})!}}\,\,
 e_sv_0^{a_{i_0}}\otimes v_{\frac{a_{i_0}+b_{j_0}-c}{2}}^{b_{j_0}} 
 \quad \\ + \quad
 \sum_{j<\frac{a_{i_0}+b_{j_0}-c}{2}} q_j\;
 e_s v_{\frac{a_{i_0}+b_{j_0}-c}{2}-j}^{a_{i_0}}\otimes v_j^{b_{j_0}} .
\end{multline*}
We know from \eqref{eq.actionV(m)Bis} 
(see also \eqref{eq.simmetryCG} and \eqref{eq.extremoCG}) 
that,
if $i_0>0$ and $s=\frac{a_{i_0}-a_{i_0-1}+m}{2}$ then 
\begin{align*}
e_s\, v_0^{a} 
& =
CG(\tfrac{a_{i_0-1}}{2},\tfrac{a_{i_0-1}}{2};\,\tfrac{a_{i_0}}{2},-\tfrac{a_{i_0}}{2}
\,|\,\tfrac{m}{2},\tfrac{m}{2}-s)\,\,v_{0}^{a_{i_0-1}}\\[2mm]
& =
\sqrt{\frac{(m+1)!\;(a_{i_0})!\;(a_{i_0-1})!}
{\big(\tfrac{a_{i_0}+a_{i_0-1}+m}{2}+1\big)!\;\big(\tfrac{a_{i_0}+a_{i_0-1}-m}{2}\big)!}}
\,\,v_{0}^{a_{i_0-1}} \ne 0.
\end{align*}
This implies $(e_sv_0)_1\ne 0$. 
\end{proof}

\begin{proposition}\label{prop:Zocalo_t}
Let 
$V=V(a_0)\oplus \hdots \oplus V(a_\ell)$ and 
$W=V(b_0)\oplus \hdots \oplus V(b_{\ell'})$,
with $\ell,\ell'\ge1$, 
be the socle decomposition of 
two uniserial $\g_m$-modules. 
If $t>\min\{ \ell,\ell'\}$ then
\[
S_t=\Big(\bigoplus_{i+j=t} V(a_{i})\otimes V(b_j)\Big)^{\mathfrak{r}} =0.
\]
If $0<t\le\min\{ \ell,\ell'\}$ and 
$\mu$ is a highest weight in 
$S_t$
then $\mu$ must be a highest weight
in all the summands $V(a_i)\otimes V(b_{t-i})$, $i=0,\dots,t$,
and, in this case, 
its weight space is 1-dimensional and generated by a linear combination 
\[
\sum_{i=0}^t q_{i}\,v_0^{a_i,b_{t-i},\mu}
\]
with $q_{i}\ne0$ for all $i=0,\dots,t$.
\end{proposition}

\begin{proof}
We fix $t> 0$ and we assume that there is a non-zero
\[
u=\sum_{i+j=t} u_{i,j}
\in \Big(\bigoplus_{i+j=t} V(a_{i})\otimes V(b_j)\Big)^{\mathfrak{r}}
\]
that is a  highest weight vector of weight $\mu$.
Since $V(a_{i})\otimes V(b_j)$ is an $\mathfrak{sl}(2)$-submodule, it follows that $u_{i,j}$ is either zero or a highest weight vector of weight $\mu$. 
Let
\[
 I_t^\mu=\{(i,j): 0\le i\le  \ell,\; 0\le j \le\ell',\;i+j=t \text{ and } u_{i,j}\ne 0\}.
\]
Since $u\ne0$, it follows that 
$I_t^\mu\ne\emptyset$ and 
\[
 u=\sum_{(i,j)\in I_t^\mu}q_{i,j}\,v_0^{a_i,b_j,\mu}
\]
for certain non-zero scalars $0\ne q_{i,j}\in\F$. 
We will show that  
$t\le\min\{ \ell,\ell'\}$ and
\begin{equation}\label{eq:It}
I_t^\mu=\{(0,t),(1,t-1),\dots,(t,0)\}.
\end{equation}
Since $u$ is $\mathfrak{r}$-invariant, we have
\begin{equation}\label{eq:linear_system}
0 =e_s u =\sum_{(i,j)\in I_t^\mu}q_{i,j} \big(
(e_s v_0^{a_i,b_j,\mu})_1  +
(e_s v_0^{a_i,b_j,\mu})_2  +
(e_s v_0^{a_i,b_j,\mu})_3  \big)
\end{equation}
 for all $0\leq s\leq m$ (see \eqref{eq.e_sv}). 
 
Since we know that 
$(e_s v_0^{a_i,b_j,\mu})_1\in V(a_{i-1})\otimes V(b_{i})$,  
$(e_s v_0^{a_i,b_j,\mu})_2 \in V(a_{i})\otimes V(b_{i-1})$,
and 
$(e_sv_0^{a_i,b_j,\mu})_3\in V(a_{i-2}) \otimes V(b_j) \oplus V(a_i)\otimes V(b_{j-2})$, 
 \eqref{eq:linear_system} yields a linear system whose unknowns are the $q_{i,j}$'s and the equations are
 \begin{align}\label{eq.system_2}
 q_{i,j} (e_s v_0^{a_i,b_j,\mu})_1 +  q_{i-1,j+1} (e_s v_0^{a_{i-1},b_{j+1},\mu})_2&=0,\quad \text{if $(i,j), (i-1,j+1)\in I_t^\mu$}; \\ 
 \label{eq.system_1}
 q_{i,j} (e_s v_0^{a_i,b_j,\mu})_1 &=0,\quad \text{if $(i,j)\in I_t^\mu$, $(i-1,j+1)\not\in I_t^\mu$};\\  
 \label{eq.system_1bis}
 q_{i,j} (e_s v_0^{a_i,b_j,\mu})_2 &=0,\quad \text{if $(i,j)\in I_t^\mu$, $(i+1,j-1)\not\in I_t^\mu$}
\end{align}
for each $s=0,\dots,m$. 
Suppose, if possible, that there is 
$(i_0,j_0)\in I_t^\mu$, $i_0>0$, such that 
$(i_0-1,j_0+1)\not\in I_t^\mu$ (this would happen if $t>\ell'$). Then \eqref{eq.system_1} implies 
$(e_s v_0^{a_i,b_j,\mu})_1 =0$ for all  $s=0,\dots,m$, 
which contradicts Lemma \ref{lemma.no nulo arriba e izquierda}. 
This proves that $t\le \ell'$ and $(0,t)\in I_t^\mu$. 

Similarly, it can be shown that if 
 $(i_0,j_0)\in I_t^\mu$, then either $j_0=0$ or 
$(i_0+1,j_0-1)\in I_t^\mu$. 
This shows that $t\le \ell$ and \eqref{eq:It}.

Finally, it follows from 
Lemma \ref{lemma.no nulo arriba e izquierda}
that 
the linear system given by \eqref{eq.system_2}, \eqref{eq.system_1} and \eqref{eq.system_1bis} 
has at most a 1-dimensional solution space. 
\end{proof}

\subsection{The socle of the tensor product of $\g_m$-modules of type $Z$}
Given two uniserial $\g_m$-modules $V_1$ and $V_2$, 
of length 2, 
with socle decomposition 
$V_1=V(a)\oplus V(b)$ and 
$V_2=V(c)\oplus V(d)$,
we know from the previous section 
that 
\[
\text{soc}(V_1 \otimes V_2)=V(a)\otimes V(c)\; \oplus S_1
\]
with 
\[
S_1=\text{soc}(V_1 \otimes V_2)\cap 
\Big(V(a)\otimes V(d)\;\oplus \;
V(b)\otimes V(d)\Big).
\]
The following theorem describes $S_1$ for
uniserials of length 2 of type $Z$ (see \S\ref{wdos}). 

 \begin{theorem}\label{thm:length_2}
 Let 
$V_1=V(a)\oplus V(b)$ and 
$V_2=V(c)\oplus V(d)$ be the socle decomposition of 
two uniserial $\g_m$-modules
of type $Z$.
Then the following table describes $S_1$:
\begin{center}

\noindent
\begin{tabular}{c|c|c|}
${}_{\displaystyle V_1}\backslash 
{\displaystyle V_2}$ & 
$\begin{array}{l}
Z(c,1) \\[1mm] 
\simeq V(c)\oplus V(c+m)
\end{array}$   & 
$\begin{array}{l}
Z(d,1)^* \\[1mm] 
\simeq V(d+m)\oplus V(d)
\end{array}$ 
\rule[-3mm]{0mm}{8mm}\\
\hline
$\begin{array}{l}
Z(a,1)\\[1mm] 
\simeq V(a)\oplus V(a+m)
\end{array}$  & 
$S_1\simeq V(a+d); $ & 
$S_1\simeq\begin{cases} 
V(d-a),& \text{if $a\le d$;} \\ 
0,& \text{if $a> d$;} \end{cases}$ 
\rule[-6mm]{0mm}{14mm}\\
\hline
$\begin{array}{l}
Z(b,1)^*\\[1mm] 
\simeq V(b+m)\oplus V(b)
\end{array}$  & 
$S_1\simeq\begin{cases} 
V(b-c),& \text{if $b\ge c$;} \\ 
0,& \text{if $b< c$;} \end{cases}$  & 
$S_1=0$.
\rule[-6mm]{0mm}{14mm}\\
\hline
\end{tabular}
\end{center}
\medskip

\noindent
All the isomorphisms in the table are as $\sl(2)$-modules.
A highest weight vector is 
\[
u_0=\sqrt{d+1}\;v_0^{a,d,\mu} - 
\sqrt{b+1}\;v_0^{b,c,\mu} 
\]
in the entries (1,1) and (1,2) of the table, 
with $\mu=a+d=b+c$ and $\mu=d-a=c-b$ respectively, 
and 
\[
u_0=\sqrt{d+1}\;v_0^{a,d,\mu}\;
-(-1)^m
\sqrt{b+1}\;   v_0^{b,c,\mu}
\]
in the entries (2,1) of the table
with $\mu=a-d=b-c$. 
 \end{theorem}

It would be very interesting for us to extend this theorem 
to any pair of uniserials of length 2. 
We have the following conjecture:

\begin{conj}\label{conj:length2}
 Let 
$V_1=V(a)\oplus V(b)$ and 
$V_2=V(c)\oplus V(d)$  be the socle decomposition of 
two uniserial $\g_m$-modules,
as in Theorem \ref{thm.CS_Classification}, 
and assume that $a<c$, or
$a=c$ and $b\le d$. 
Then $S_1=0$ except in the following cases. 
\begin{enumerate}[\hspace{3mm}]
\item[\tiny$\bullet$]
Case 1: $[a,b]=[0,m]$. Here $S_1\simeq V(d)$. 
\item[\tiny$\bullet$] Cases 2: Here $a>0$.
\begin{enumerate}[\hspace{3mm}]
\item[--] Case 2.1: $a+b=c+d=m$ with $d-a=b-c\ge0$. 
Here $S_1\simeq V(d-a)$.
\item[--] Case 2.2: $b-a=d-c=m$. Here $S_1\simeq V(d+a)$.
\item[--] Case 2.3: $b-a=c-d=m$ with $d-a=c-b\ge0$. 
Here $S_1\simeq V(d-a)$.
 \end{enumerate}
\item[\tiny$\bullet$] Case 3: $[c,d]=[b,a]$. Here $S_1\simeq V(0)$.
 \end{enumerate}
 \end{conj}

Note that the entries (1,1) and (1,2) in the table of 
 Theorem \ref{thm:length_2}
 correspond to Cases 2.2 and 2.3 respectively,
 while the entry (2,1), with $b\ge c$, 
 is ruled out in the conjecture by the condition 
 $a\le c$.

\begin{proof}[Proof of Theorem \ref{thm:length_2}]
Let $\mu$ be a possible highest weight in $S_1$.
We first point out some general considerations that will
be useful for all cases, and next we will 
work out the details of each case. 

We know from Proposition \ref{prop:Zocalo_t}
that $\mu$ must be highest weight in both 
$V(a)\otimes V(d)$ and $V(b)\otimes V(c)$, that is 
\begin{equation}\label{eq:cotas_mu}
|a-d|,|b-c|\le \mu \le a+d,b+c
\end{equation}
and $\mu\equiv a+d\equiv b+c\mod 2$.
We also know that $\mu$ is indeed
 highest weight in $S_1$ if and only if
 there is a linear combination 
\[
u_0=q_1v_0^{a,d,\mu}+q_2 v_0^{b,c,\mu},
\]
with $q_1,q_2\ne 0$, that is 
annihilated by $e_s$ for all 
$s=0,\dots,m$.
We now describe 
$e_sv_0^{a,d,\mu}$ and 
$e_sv_0^{b,c,\mu}$.

On the one hand we have (see \eqref{eq.Vc_en_tensor})
\begin{equation*}
v_0^{a,d,\mu}
=\sum_{i,j} 
CG(\tfrac{a}{2},\tfrac{a}{2}-i;\,\tfrac{d}{2},\tfrac{d}{2}-j
\,|\,\tfrac{\mu}{2},\tfrac{\mu}{2})\,
 v_i^a\otimes v_j^d 
\end{equation*}
and thus (see \eqref{eq.actionV(m)})
\begin{align}
e_sv_0^{a,d,\mu}
& =\sum_{i,j} 
CG(\tfrac{a}{2},\tfrac{a}{2}-i;\,\tfrac{d}{2},\tfrac{d}{2}-j
\,|\,\tfrac{\mu}{2},\tfrac{\mu}{2})\, v_{i}^a\otimes e_sv_{j}^d \notag \\
& =\sum_{i,j,k} (-1)^j
CG(\tfrac{a}{2},\tfrac{a}{2}-i;\,\tfrac{d}{2},\tfrac{d}{2}-j
\,|\,\tfrac{\mu}{2},\tfrac{\mu}{2})\notag \\
&\hspace{4cm}\times 
CG(\tfrac{c}{2},\tfrac{c}{2}-k;\,\tfrac{d}{2},-\tfrac{d}{2}+j
\,|\,\tfrac{m}{2},\tfrac{m}{2}-s)\,
v_{i}^a\otimes v_{k}^{c} \notag \\
& =\sum_{i,j,k} (-1)^k
CG(\tfrac{a}{2},\tfrac{a}{2}-i;\,\tfrac{d}{2},\tfrac{d}{2}-k
\,|\,\tfrac{\mu}{2},\tfrac{\mu}{2})\notag \\\label{eq:esvad}
&\hspace{4cm}\times 
CG(\tfrac{c}{2},\tfrac{c}{2}-j;\,\tfrac{d}{2},-\tfrac{d}{2}+k
\,|\,\tfrac{m}{2},\tfrac{m}{2}-s)\,
v_{i}^a\otimes v_{j}^{c}.
\end{align}
In this sum, if the coefficient of $v_{i}^a\otimes v_{j}^{c}$ is not zero then we must have
\begin{equation}\label{eq.Coef_no_nulo_ad}
\begin{split}
\frac{a}{2}-i+\frac{d}{2}-k & = \frac{\mu}{2}, \\
\frac{c}{2}-j-\frac{d}{2}+k & =\frac{m}{2}-s.
\end{split}
\end{equation}

On the other hand we have (see \eqref{eq.Vc_en_tensor})
\begin{equation*}
v_0^{b,c,\mu}
 =\sum_{i,j} 
CG(\tfrac{b}{2},\tfrac{b}{2}-i;\,\tfrac{c}{2},\tfrac{c}{2}-j
\,|\,\tfrac{\mu}{2},\tfrac{\mu}{2})\,
 v_i^b\otimes v_j^c. 
\end{equation*}
and thus (see \eqref{eq.actionV(m)})
\begin{align}
e_sv_0^{b,c,\mu}
& =\sum_{i,j} 
CG(\tfrac{b}{2},\tfrac{b}{2}-i;\,\tfrac{c}{2},\tfrac{c}{2}-j
\,|\,\tfrac{\mu}{2},\tfrac{\mu}{2})\,
 e_s v_i^b\otimes v_j^c. \notag \\
& =\sum_{i,j,k} (-1)^i
CG(\tfrac{b}{2},\tfrac{b}{2}-i;\,\tfrac{c}{2},\tfrac{c}{2}-j
\,|\,\tfrac{\mu}{2},\tfrac{\mu}{2})\notag \\
&\hspace{4cm}\times  
CG(\tfrac{a}{2},\tfrac{a}{2}-k;\,\tfrac{b}{2},-\tfrac{b}{2}+i
\,|\,\tfrac{m}{2},\tfrac{m}{2}-s)\,
v_k^a\otimes v_j^c \notag \\
& =\sum_{i,j,k} (-1)^k 
CG(\tfrac{b}{2},\tfrac{b}{2}-k;\,\tfrac{c}{2},\tfrac{c}{2}-j
\,|\,\tfrac{\mu}{2},\tfrac{\mu}{2})\notag \\\label{eq:esvbc}
&\hspace{4cm}\times  
CG(\tfrac{a}{2},\tfrac{a}{2}-i;\,\tfrac{b}{2},-\tfrac{b}{2}+k
\,|\,\tfrac{m}{2},\tfrac{m}{2}-s)\,
v_i^a\otimes v_j^c.
\end{align}
In this sum, if the coefficient of $v_{i}^a\otimes v_{j}^{c}$ is not zero then we must have
\begin{equation}\label{eq.Coef_no_nulo_bc}
\begin{split}
\frac{b}{2}-k+\frac{c}{2}-j & = \frac{\mu}{2}, \\
\frac{a}{2}-i-\frac{b}{2}+k & =\frac{m}{2}-s.
\end{split}
\end{equation}
Either \eqref{eq.Coef_no_nulo_ad} or  \eqref{eq.Coef_no_nulo_bc} imply
\begin{equation}\label{eq.imasj}
i+j=\frac{a+c-m-\mu}{2}+s,
\end{equation}
and recall that $0\le i\le a$ and $0\le j \le c$.

\bigskip

\noindent
\emph{The case $Z(a,1)\otimes Z(c,1)$}.
Here $b=a+m$, $d=c+m$ and \eqref{eq:cotas_mu} implies 
\begin{equation}\label{eq:cota_p_ZZ}
\mu=a+c+m-2p,\quad 0\le p\le \min\{a,c+m\}.
\end{equation}
It follows from  \eqref{eq.imasj}
that
\begin{equation}\label{eq:imasj_ZZ}
0\le i+j=p-m+s.
\end{equation}

First we prove that if $p=0$ then $\mu$ is indeed
a highest weight in $S_1$.
In this case, it follows from \eqref{eq:imasj_ZZ}
that 
\begin{equation*}
e_sv_0^{a,d,\mu}= e_sv_0^{b,c,\mu}=0
\end{equation*}
for all $s=0,\dots,m-1$. 

For $s=m$, in the sums describing 
$e_mv_0^{a,d,\mu}$ and $e_mv_0^{b,c,\mu}$
we must have $i+j=0$, that is $i=j=0$, and thus
(see \eqref{eq:esvad} and \eqref{eq:esvbc})
\begin{align*}
e_mv_0^{a,d,\mu}
&=
CG(\tfrac{a}{2},\tfrac{a}{2};\,
\tfrac{c+m}{2},\tfrac{c+m}{2}
\,|\,\tfrac{\mu}{2},\tfrac{\mu}{2})\, 
CG(\tfrac{c}{2},\tfrac{c}{2};\,
\tfrac{c+m}{2},-\tfrac{c+m}{2}
\,|\,\tfrac{m}{2},-\tfrac{m}{2})\,
v_{0}^a\otimes v_{0}^{c} \\
&= \sqrt{\frac{m+1}{c+m+1}}\;v_{0}^a\otimes v_{0}^{c} 
\end{align*}
and 
\begin{align*}
e_mv_0^{b,c,\mu}
&=
CG(\tfrac{a+m}{2},\tfrac{a+m}{2};\,
\tfrac{c}{2},\tfrac{c}{2}
\,|\,\tfrac{\mu}{2},\tfrac{\mu}{2})\, 
CG(\tfrac{a}{2},\tfrac{a}{2};\,
\tfrac{a+m}{2},-\tfrac{a+m}{2}
\,|\,\tfrac{m}{2},-\tfrac{m}{2})\,
v_0^a\otimes v_0^c \\
&= \sqrt{\frac{m+1}{a+m+1}}\;
v_0^a\otimes v_0^c.
\end{align*}
This implies that 
\[
u_0=\sqrt{c+m+1}\;v_0^{a,d,\mu} - 
\sqrt{a+m+1}\;v_0^{b,c,\mu} 
\]
is, indeed, a highest weight vector, 
of weight $\mu=a+c+m$, in $S_1$. 

\medskip

We now prove that if $p\ge 1$ then $\mu=a+c+m-2p$ is not a highest weight in $S_1$.

Let us fix $s=m$. It follows from
\eqref{eq.Coef_no_nulo_ad} and \eqref{eq:imasj_ZZ} that 
\begin{equation*}
k = j= p-i
\end{equation*}
and (see \ref{eq:cota_p_ZZ})
\begin{align*}
e_mv_0^{a,d,\mu}
& =\sum_{i=0}^{\min\{p,a\}} (-1)^{p-i}
CG(\tfrac{a}{2},\tfrac{a}{2}-i;\,\tfrac{c+m}{2},\tfrac{c+m}{2}+i-p
\,|\,\tfrac{a+c+m}{2}-p,\tfrac{a+c+m}{2}-p)\, \\
&
\hspace{2cm}\times 
CG(\tfrac{c}{2},\tfrac{c}{2}+i-p;\,\tfrac{c+m}{2},-\tfrac{c+m}{2}+p-i
 \,|\,\tfrac{m}{2},-\tfrac{m}{2})\,
v_{i}^a\otimes v_{p-i}^{c} \\[5mm]
& =\sum_{i=0}^{\min\{p,a\}}  (-1)^{i}
\sqrt{\frac{(a+c+m-2p+1)!\; p!\;(a-i)!\;(c+m-p+i)!}
{(a+c+m-p+1)!\;(a-p)!\;(c+m-p)!\;i!\;(p-i)!}} \\
&
\hspace{4cm}\times 
\sqrt{\frac{(m+1)!\;c!\;(c+m+i-p)!}
{(c+m+1)!\;m!\;(c+i-p)!}}\;\;
v_{i}^a\otimes v_{p-i}^{c}.
\end{align*}
This is, up to a non-zero scalar, equal to
\[
w^{a,d,\mu}=
\sum_{i=0}^{\min\{p,a\}}  (-1)^{i}
\sqrt{\frac{(a-i)!\;(c+m+i-p)!^2}
{\;i!\;(p-i)!\;(c+i-p)!}}\;\;
v_{i}^a\otimes v_{p-i}^{c}.
\]

Similarly, it follows from
\eqref{eq.Coef_no_nulo_bc} and \eqref{eq:esvbc} that \begin{equation*}
k = i,\quad j= p-i
\end{equation*}
and (see \ref{eq:cota_p_ZZ})
\begin{align*}
e_mv_0^{b,c,\mu}
& =\sum_{i=0}^{\min\{p,a\}} (-1)^i \;
CG(\tfrac{a+m}{2},\tfrac{a+m}{2}-i;\,\tfrac{c}{2},\tfrac{c}{2}+i-p
\,|\,\tfrac{a+c-m}{2}-p,\tfrac{a+c-m}{2}-p) \\
&
\hspace{2cm}\times 
CG(\tfrac{a}{2},\tfrac{a}{2}-i;\,\tfrac{a+m}{2},-\tfrac{a+m}{2}+i
\,|\,\tfrac{m}{2},-\tfrac{m}{2})\,
v_i^a\otimes v_{p-i}^c \\[3mm]
& =\sum_{i=0}^{\min\{p,a\}}  (-1)^{i} 
\sqrt{\frac{(a+c+m-2p+1)!\;p!\;(a+m-i)!\;(c+i-p)!}
{(a+c+m-p+1)!\;(a+m-p)!\;(c-p)!\;i!\;(p-i)!}} \\
&
\hspace{4cm}\times 
\sqrt{\frac{(m+1)!\;a!\;(a+m-i)!}
{(a+m+1)!\;m!\;(a-i)!}}\;\;
v_i^a\otimes v_{p-i}^c, 
\end{align*}
and this is, up to a scalar, equal to
\[
w^{b,c,\mu}=\sum_{i=0}^{\min\{p,a\}}  (-1)^{i} 
\sqrt{\frac{(a+m-i)!^2\;(c+i-p)! }
{i!\;(p-i)!\;(a-i)!}}\;\;
v_i^a\otimes v_{p-i}^c.
\] 
We will show that $w^{a,d,\mu}$ and $w^{b,c,\mu}$ 
are linearly independent. 
It suffices to show that the ratio of the first two coefficients 
(corresponding to $i=0,1$) of $w^{a,d,\mu}$ and $w^{b,c,\mu}$ differ from each other. 
(We recall that $p\ge 1$ and hence $a\ge 1$ (see \ref{eq:cota_p_ZZ}). Note also that 
all the coefficients in both
$w^{a,d,\mu}$ and $w^{b,c,\mu}$ are non-zero.)

The ratio of the first two coefficients 
of $w^{a,d,\mu}$ is
\[ 
-
\frac{
\sqrt{\dfrac{a!\;(c+m-p)!^2}
{p!\;(c-p)!}}
}
{
\sqrt{\dfrac{(a-1)!\;(c+m-p+1)!^2}
{(p-1)!\;(c+1-p)!}}
}
=
-
\sqrt{\dfrac{a\;(c+1-p)}
{p\;(c+m-p+1)^2}},
\]
and ratio of the first two coefficients 
of $w^{b,c,\mu}$ is
\[ 
-
\frac{
\sqrt{\dfrac{(a+m)!^2\;(c-p)! }
{p!\;a!}}
}
{
\sqrt{\dfrac{(a+m-1)!^2\;(c+1-p)! }
{(p-1)!\;(a-1)!}} 
}
=
-
\sqrt{\dfrac{(a+m)^2}
{p\;a\;(c+1-p)}
}.
\]
This two ratios are different since
$
a^2(c+1-p)^2<(a+m)^2(c+m-p+1)^2.
$
This completes the proof of this case.

\bigskip

\noindent
\emph{The case $Z(b,1)^*\otimes Z(d,1)^*$}.
Here $a=b+m$, $c=d+m$,  and \eqref{eq:cotas_mu} implies
\[
\mu=b+d+m-2p,\quad 0\le p\le \min\{b,d\}.
\] 
We know from  \eqref{eq.imasj} that 
\begin{equation*}
0\le i+j=p+s.
\end{equation*}
Recall that \eqref{eq.Coef_no_nulo_ad} implies $k=p-i\ge 0$ and 
\eqref{eq:esvad} says
\begin{align*}
e_sv_0^{a,d,\mu}
& =\sum_{i=0}^{\min\{a,p\}} (-1)^{p-i}\;
CG(\tfrac{b+m}{2},\tfrac{b+m}{2}-i;\,\tfrac{d}{2},\tfrac{d}{2}-p+i
\,|\,\tfrac{b+d+m}{2}-p,\tfrac{b+d+m}{2}-p)\notag \\
&\hspace{1.3cm}\times 
CG(\tfrac{d+m}{2},\tfrac{d+m}{2}-p-s+i;\,\tfrac{d}{2},-\tfrac{d}{2}+p-i
\,|\,\tfrac{m}{2},\tfrac{m}{2}-s)\,
v_{i}^a\otimes v_{p+s-i}^{c} \\[3mm]
&  =\sum_{i=0}^{\min\{a,p\}} (-1)^{i}
\sqrt{
\frac{
(b+m+d-2p+1)!\;p!\;(b+m-i)!\;(d-p+i)! 
}{
(b+m+d-p+1)!\; (b+m-p)!\; (d-p) !\; i!\; (p-i)!
}
}  \\ 
&\hspace{2cm}\times
\sqrt{
\frac{
(d+m-p-s+i)!\;(p+s-i)!\;(m+1)!\;d!\;
}{
(d+m+1)!\;(p-i)!\;(d-p+i)!\;(m-s)!\;s!
}
} \;\;
v_{i}^a\otimes v_{p+s-i}^{c}.
\end{align*}
As always, the reader should check that all
the numbers under the factorial sign are non-negative. 

On the other hand, \eqref{eq.Coef_no_nulo_bc} implies 
$k=i-s\ge 0$ and 
\eqref{eq:esvbc} says
\begin{align*}
e_sv_0^{b,c,\mu}
&=\sum_{i=s}^{\min\{a,p+s\}} (-1)^{i-s} \\
& \hspace{1.5cm} \times
CG(\tfrac{b}{2},\tfrac{b}{2}-i+s;\,
\tfrac{d+m}{2},\tfrac{d+m}{2}-p-s+i
\,|\,\tfrac{b+d+m}{2}-p,\tfrac{b+d+m}{2}-p)\,   \\ 
&\hspace{2.5cm}\times
CG(\tfrac{b+m}{2},\tfrac{b+m}{2}-i;\,
\tfrac{b}{2},-\tfrac{b}{2}+i-s
\,|\,\tfrac{m}{2},\tfrac{m}{2}-s)\;\;
v_i^a\otimes v_{p+s-i}^c \\[3mm]
& =\sum_{i=s}^{\min\{a,p+s\}} (-1)^{i-s} \\
& \hspace{1.4cm}\times \sqrt{
\frac{
(b+m+d-2p+1)!\;p!\;(b-i+s)!\;(d+m-p-s+i)! 
}{
(b+m+d-p+1)!\; (b-p)!\; (d+m-p)!\; (i-s)!\; (p+s-i)!
}
}  \\ 
&\hspace{2.2cm}\times
\sqrt{
\frac{
(b+m-i)!\;i!\;(m+1)!\;b!\;
}{
(b+m+1)!\;(i-s)!\;(b-i+s)!\;(m-s)!\;s!
}
} \;\; 
v_i^a\otimes v_{p+s-i}^c.
\end{align*}
Again, the reader should check that all
the numbers under the factorial sign are non-negative. 

Now, for $s=1$, 
the sum describing 
$e_1v_0^{a,d,\mu}$ starts at $i=0$ with non-zero coefficient, while  the sum describing 
$e_1v_0^{b,c,\mu}$ starts at $i=1$.
This proves that  
$\{e_1v_0^{a,d,\mu},e_1v_0^{b,c,\mu}\}$ is linearly independent
 and hence 
there is no possible $\mu$ in $S_1$, that is $S_1=0$.
This completes the proof in this case.

\bigskip

\noindent
\emph{The case $Z(a,1)\otimes Z(d,1)^*$}.
Here $b=a+m$, $c=d+m$ and we first assume 
\[
a\le d.
\] 
In this case 
\[
\mu=d-a+2p,\quad 0\le p\le a,
\] 
and it follows from  \eqref{eq.imasj}
that
\begin{equation*}
0\le i+j=a-p+s.
\end{equation*}
Recall that \eqref{eq:esvad} says
\begin{multline*}
e_sv_0^{a,d,\mu}
=\sum_{i,j,k} (-1)^k
CG(\tfrac{a}{2},\tfrac{a}{2}-i;\,
\tfrac{d}{2},\tfrac{d}{2}-k
\,|\,\tfrac{d-a}{2}+p,\tfrac{d-a}{2}+p)\, \\ 
\times
CG(\tfrac{d+m}{2},\tfrac{d+m}{2}-j;\,
\tfrac{d}{2},-\tfrac{d}{2}+k
\,|\,\tfrac{m}{2},\tfrac{m}{2}-s)\;\;
v_{i}^a\otimes v_{j}^{c}.
\end{multline*}
It follows from \eqref{eq.Coef_no_nulo_ad}
that 
\begin{align*}
k &= a-p -i, \\
j & = a-p +s -i,
\end{align*}
and the condition $k\ge 0$ implies $i\le a-p$. 
Hence 
\begin{align*}
e_sv_0^{a,d,\mu}
&  =\sum_{i=0}^{a-p} (-1)^{a-p-i}
CG(\tfrac{a}{2},\tfrac{a}{2}-i;\,
\tfrac{d}{2},\tfrac{d}{2}-a+p+i
\,|\,\tfrac{d-a}{2}+p,\tfrac{d-a}{2}+p)\, \\ 
&\hspace{1.3cm}\times
CG(\tfrac{d+m}{2},\tfrac{d+m}{2}-a+p-s+i;\,
\tfrac{d}{2},-\tfrac{d}{2}+a-p-i
\,|\,\tfrac{m}{2},\tfrac{m}{2}-s) \\
&\hspace{9cm}\times\;
v_{i}^a\otimes v_{a-p+s-i}^{c} \\[3mm]
&  =\sum_{i=0}^{a-p} (-1)^{i}
\sqrt{
\frac{
(d-a+2p+1)!\;(a-p)!\;(a-i)!\;(d-a+p+i)! 
}{
(d+p+1)!\; p!\; (d-a+p) !\; i!\; (a-p-i)!
}
}  \\ 
&\hspace{1.3cm}\times
\sqrt{
\frac{
(d+m-a+p-s+i)!\;(a-p+s-i)!\;(m+1)!\;d!\;
}{
(d+m+1)!\;(a-p-i)!\;(d-a+p+i)!(m-s)!\;s!
}
}\\
&\hspace{9cm}\times\;
v_{i}^a\otimes v_{a-p+s-i}^{c} .
\end{align*}
At this point, the reader should check that all
the numbers under the factorial sign are non-negative. 
This last sum is, up to the non-zero scalar
\[
\sqrt{\frac{1}{s!\;(m-s)!}}\;
\sqrt{
\frac{
(d-a+2p+1)!\;(a-p)!\;(m+1)!\;d!
}{
(d+p+1)!\; p!\; (d-a+p) !\; (d+m+1)!
}
}  
,\]
equal to
\[
w^{a,d,\mu}_s=
\sum_{i=0}^{a-p} (-1)^{i}
\sqrt{
\frac{
(a-i)!\;(d+m-a+p-s+i)!\;(a-p+s-i)!
}{
i!\; (a-p-i)!^2
}
} \;\;
v_{i}^a\otimes v_{a-p-i}^{c}.
\]

On the other hand, \eqref{eq:esvbc} says
\begin{multline*}
e_sv_0^{b,c,\mu}
=\sum_{i,j,k} (-1)^k 
CG(\tfrac{b}{2},\tfrac{b}{2}-k;\,
\tfrac{c}{2},\tfrac{c}{2}-j
\,|\,\tfrac{\mu}{2},\tfrac{\mu}{2})\,   \\ 
\times
CG(\tfrac{a}{2},\tfrac{a}{2}-i;\,
\tfrac{b}{2},-\tfrac{b}{2}+k
\,|\,\tfrac{m}{2},\tfrac{m}{2}-s)\;\;
v_i^a\otimes v_j^c 
\end{multline*}
and it follows from 
\eqref{eq.Coef_no_nulo_bc}
that 
\begin{align*}
j & = a-p+s-i, \\
k & =m+i-s,
\end{align*}
and the condition $j\ge 0$ implies $i\le a-p+s$.
Thus 
\begin{align*}
e_sv_0^{b,c,\mu}
& =\sum_{i=0}^{\min\{a,a-p+s\}} (-1)^{m+i-s} \\
& \hspace{1cm}\times 
CG(\tfrac{a+m}{2},\tfrac{a-m}{2}-i+s;\,
\tfrac{d+m}{2},\tfrac{d+m}{2}-a+p-s+i
\,|\,\tfrac{d-a}{2}+p,\tfrac{d-a}{2}+p)\,   \\[1mm] 
&\hspace{2cm}\times
CG(\tfrac{a}{2},\tfrac{a}{2}-i;\,
\tfrac{a+m}{2},-\tfrac{a-m}{2}+i-s
\,|\,\tfrac{m}{2},\tfrac{m}{2}-s)\;\;
v_i^a\otimes v_{a-p+s-i}^c \\[3mm]
& =\sum_{i=0}^{\min\{a,a-p+s\}} (-1)^{i} \\
& \hspace{1.2cm}\times 
\sqrt{
\frac{
(d-a+2p+1)!\,(m+a-p)!\,(a-i+s)!\,(d+m-a+p+i-s)!
}{
(m+d+p+1)!\,p!\,(d-a+p)!\,(m+i-s)!\,(a-p-i+s)!
}
}  \\ 
&\hspace{2cm}\times
\sqrt{
\frac{ 
(a-i+s)!\;(m+i-s)!\,a!\,(m+1)!\,
}{
(m+1+a)!\,(a-i)!\,i!\;s!\;(m-s)!
}
} \;\;
v_i^a\otimes v_{a-p+s-i}^c.
\end{align*}
As above, at this point, the reader should check that all
the numbers under the factorial sign are non-negative. 
The above sum is, up to the non-zero scalar
\[
\sqrt{\frac{1}{s!\;(m-s)!}}\;
\sqrt{
\frac{
(d-a+2p+1)!\,(m+a-p)!\;(m+1)!\,a!
}{
(m+d+p+1)!\,p!\,(d-a+p)!\;(m+1+a)!
}
}, 
\]
equal to
\[
w^{b,c,\mu}_s=
\sum_{i=0}^{\min\{a,a-p+s\}} (-1)^{i}
\sqrt{
\frac{
(a-i+s)!^2\,(d+m-a+p+i-s)!\;
}{
(a-p-i+s)!\,(a-i)!\;i!
}
} \;\;
v_i^a\otimes v_{a-p+s-i}^c.
\]
If $p=0$ then 
\[
w^{a,d,\mu}_s=w^{b,c,\mu}_s=
\sum_{i=0}^{a} (-1)^{i}
\sqrt{
\frac{
(a-i+s)!\,(d+m-a+i-s)!\;
}{
(a-i)!\;i!
}
} \;\;
v_i^a\otimes v_{a+s-i}^c
\]
for all $s=0,\dots,m$. 
This shows that 
\[
u_0=\sqrt{d+1}\;v_0^{a,d,\mu}\;
-
\sqrt{b+1}\;   v_0^{b,c,\mu}
\]
is, indeed, a highest weight vector, 
of weight $\mu=d-a$, 
in $S_1$. 

On the other hand, assume $p\ge 1$.
Then, for $s=1$, the sum defining 
$w^{b,c,\mu}_1$ has the index $i$ running 
up to $i=a+1-p$ while 
in the sum defining 
$w^{a,d,\mu}_1$ the index $i$ only runs up to $i=a-p$.
In both cases, 
all the coefficients are non-zero, and thus 
$\{w^{a,d,\mu}_1, w^{b,c,\mu}_1\}$ is linearly independent. 
This completes the proof in the case $d\ge a$

\medskip

We now assume 
\[
a>d.
\]
In this case 
\[
\mu=a-d+2p,\quad 0\le p\le d,
\] 
and it follows from  \eqref{eq.imasj}
that
\begin{equation*}
0\le i+j=d-p+s.
\end{equation*}

From \eqref{eq:esvad} we have
\begin{multline*}
e_sv_0^{a,d,\mu}
=\sum_{i,j,k} (-1)^k
CG(\tfrac{a}{2},\tfrac{a}{2}-i;\,
\tfrac{d}{2},\tfrac{d}{2}-k
\,|\,\tfrac{a-d}{2}+p,\tfrac{a-d}{2}+p)\, \\ 
\times
CG(\tfrac{d+m}{2},\tfrac{d+m}{2}-j;\,
\tfrac{d}{2},-\tfrac{d}{2}+k
\,|\,\tfrac{m}{2},\tfrac{m}{2}-s)\;\;
v_{i}^a\otimes v_{j}^{c}.
\end{multline*}
It follows from \eqref{eq.Coef_no_nulo_ad}
that 
\begin{align*}
j & = d-p +s -i, \\
k &= d-p -i.
\end{align*}
and the condition $k\ge 0$ implies $i\le d-p$. 
Hence 
\begin{align*}
e_sv_0^{a,d,\mu}
&  =\sum_{i=0}^{d-p} (-1)^{d-p-i}
CG(\tfrac{a}{2},\tfrac{a}{2}-i;\,
\tfrac{d}{2},\tfrac{d}{2}-d+p+i
\,|\,\tfrac{a-d}{2}+p,\tfrac{a-d}{2}+p)\, \\ 
&\hspace{1.3cm}\times
CG(\tfrac{d+m}{2},\tfrac{d+m}{2}-d+p-s+i;\,
\tfrac{d}{2},-\tfrac{d}{2}+d-p-i
\,|\,\tfrac{m}{2},\tfrac{m}{2}-s) \\
&\hspace{8.8cm}\times\;
v_{i}^a\otimes v_{d-p+s-i}^{c} \\[3mm]
&  =\sum_{i=0}^{d-p} (-1)^{i}
\sqrt{
\frac{
(a-d+2p+1)!\;(d-p)!\;(a-i)!\;(p+i)! 
}{
(a+p+1)!\; p!\; (a-d+p) !\; i!\; (d-p-i)!
}
}  \\ 
&\hspace{1.7cm}\times
\sqrt{
\frac{
(m+p-s+i)!\;(d-p+s-i)!\;(m+1)!\;d!\;
}{
(d+m+1)!\;(d-p-i)!\;(p+i)!(m-s)!\;s!
}
} \;\;
v_{i}^a\otimes v_{d-p+s-i}^{c}.
\end{align*}
At this point, the reader should check that all
the numbers under the factorial sign are non-negative. 
This last sum is, up to the non-zero scalar
\[
\sqrt{
\frac{
(a-d+2p+1)!\;(d-p)!\;(m+1)!\;d!
}{
(a+p+1)!\; p!\; (a-d+p) !\; (d+m+1)!\;(m-s)!\;s!
}
}  
,\]
equal to
\[
w^{a,d,\mu}_s=
\sum_{i=0}^{d-p} (-1)^{i}
\sqrt{
\frac{
(a-i)!\;(m+p-s+i)!\;(d-p+s-i)!
}{
i!\; (d-p-i)!^2
}
} \;\;
v_{i}^a\otimes v_{d-p+s-i}^{c}.
\]

On the other hand, recall that \eqref{eq:esvbc} is
\begin{multline*}
e_sv_0^{b,c,\mu}
=\sum_{i,j,k} (-1)^k 
CG(\tfrac{b}{2},\tfrac{b}{2}-k;\,
\tfrac{c}{2},\tfrac{c}{2}-j
\,|\,\tfrac{\mu}{2},\tfrac{\mu}{2})\,   \\ 
\times
CG(\tfrac{a}{2},\tfrac{a}{2}-i;\,
\tfrac{b}{2},-\tfrac{b}{2}+k
\,|\,\tfrac{m}{2},\tfrac{m}{2}-s)\;\;
v_i^a\otimes v_j^c 
\end{multline*}
and it follows from 
\eqref{eq.Coef_no_nulo_bc}
that 
\begin{align*}
j & = d-p+s-i, \\
k & =m+i-s,
\end{align*}
and the condition $j\ge 0$ implies $i\le d-p+s$.
Thus 
\begin{align*}
e_sv_0^{b,c,\mu}
& =\sum_{i=0}^{\min\{a,d-p+s\}} (-1)^{m+i-s} \\
& \hspace{1.4cm}\times 
CG(\tfrac{a+m}{2},\tfrac{a-m}{2}-i+s;\,
\tfrac{d+m}{2},\tfrac{m-d}{2}+p-s+i
\,|\,\tfrac{a-d}{2}+p,\tfrac{a-d}{2}+p)\,   \\[1mm] 
&\hspace{2cm}\times
CG(\tfrac{a}{2},\tfrac{a}{2}-i;\,
\tfrac{a+m}{2},-\tfrac{a-m}{2}+i-s
\,|\,\tfrac{m}{2},\tfrac{m}{2}-s)\;\;
v_i^a\otimes v_{d-p+s-i}^c \\[3mm]
& =\sum_{i=0}^{\min\{a,d-p+s\}} (-1)^{i} \\
& \hspace{1.4cm}\times 
\sqrt{
\frac{
(a-d+2p+1)!\,(m+d-p)!\,(a-i+s)!\,(m+p+i-s)!
}{
(m+a+p+1)!\,p!\,(a-d+p)!\,(m+i-s)!\,(d-p-i+s)!
}
}  \\ 
&\hspace{2cm}\times
\sqrt{
\frac{ 
(a-i+s)!\;(m+i-s)!\,a!\,(m+1)!\,
}{
(m+1+a)!\,(a-i)!\,i!\;s!\;(m-s)!
}
} \;\;
v_i^a\otimes v_{d-p+s-i}^c.
\end{align*}
As above, at this point, the reader should check that all
the numbers under the factorial sign are non-negative. 
The above sum is, up to the non-zero scalar
\[
\sqrt{
\frac{
(a-d+2p+1)!\,(m+d-p)!\;(m+1)!\,a!
}{
(m+a+p+1)!\,p!\,(a-d+p)!\;(m+1+a)!\,s!\;(m-s)!
}
}, 
\]
equal to
\[
w^{b,c,\mu}_s=
\sum_{i=0}^{\min\{a,d-p+s\}} (-1)^{i}
\sqrt{
\frac{
(a-i+s)!^2\,(m+p+i-s)!
}{
(d-p-i+s)!\,(a-i)!\;i!
}
} \;\;
v_i^a\otimes v_{d-p+s-i}^c.
\]
Since $a>d$, for $s=1$, 
the sum defining 
$w^{b,c,\mu}_1$ has the index $i$ running up to 
$i=d+1-p$ while 
the sum defining 
$w^{a,d,\mu}_1$ has the index $i$ running only 
up to $i=d-p$.
In both cases, 
all the coefficients are non-zero, and thus 
$\{w^{a,d,\mu}_1, w^{b,c,\mu}_1\}$ is linearly independent. 
This shows that there is no possible $\mu$ in $S_1$ and thus $S_1=0$.
This completes the proof in this case.

\medskip

Since the case $Z(b,1)^* \otimes Z(c,1)$ is derived from the 
case $Z(a,1)\otimes Z(d,1)^*$, we have completed the proof
of the theorem. We warn the reader that in order to 
obtain 
the highest weight vector in this case from 
the case
$Z(a,1)\otimes Z(d,1)^*$, 
it is needed to swap the tensor factors to go 
from $v_0^{d,a,\mu}$ and $v_0^{c,b,\mu}$
to $v_0^{a,d,\mu}$ and $v_0^{b,c,\mu}$ respectively.
To do this it is needed \eqref{eq:swap}.
\end{proof}

 \begin{theorem}\label{thm:main}
Let 
$V=V(a_0)\oplus \hdots \oplus V(a_\ell)$ and 
$W=V(b_0)\oplus \hdots \oplus V(b_{\ell'})$
be socle decomposition of two uniserial $\g_m$-modules of type $Z$.
Then
\[
\text{soc}(V \otimes W)=\text{soc}(V)\otimes \text{soc}(W)\;\oplus\; \bigoplus_{t=1}^{\min\{\ell,\ell'\}}S_t
\] 
where each $S_t$ is as follows:
\begin{enumerate}[(i)]
\item For 
$V=Z(a_0,\ell)$ and 
$W=Z(b_0,\ell')$
we have 
$S_t\simeq V(a_0+b_0+mt)$ and thus 
\[
\text{soc}(V \otimes W)\simeq 
\bigoplus_{k=0}^{\min\{a_0,b_0\}}
V(a_0+b_0-2k)\;\;
\oplus\;\;
\bigoplus_{t=1}^{\min\{\ell,\ell'\}}
V(a_0+b_0+mt).
\]

\medskip

\item For 
$V=Z(a_0,\ell)$ and 
$W=Z(b_{\ell'},\ell')^*$
we have 
\[
S_t\simeq
\begin{cases}
 0, & \text{if $a_0> b_t$};\\
 V(b_0-a_0-tm), & \text{if $a_0\le b_t$};
\end{cases}
\]
(note that $b_t=b_0-tm$) 
and thus 
\[
\text{soc}(V \otimes W)\simeq 
\begin{cases}
 \displaystyle\bigoplus_{k=0}^{b_0}
V(a_0-b_0+2k), & \text{if $a_0> b_0$};\\[5mm]
 \displaystyle\bigoplus_{k=0}^{a_0}
V(b_0-a_0+2k)\;\;
\oplus\;\;
\bigoplus_{t=1}^{T}
V(b_0-a_0-tm), & \text{if $a_0\le b_0$},
\end{cases}
\]
with $T=\min\left\{\ell,\ell',\left\lfloor\frac{b_0-a_0}{m}\right\rfloor\right\}$.

\medskip

\item For 
$V=Z(a_{\ell},\ell)^*$ and 
$W=Z(b_{\ell'},\ell')^*$
we have $S_t=0$ for all $t\ge 1$, and thus 
\[
\text{soc}(V \otimes W)\simeq
\bigoplus_{k=0}^{\min\{a_0,b_0\}}
V(|a_0-b_0|+2k).
\]
\end{enumerate}
In particular, $\text{soc}(V \otimes W)$ is multiplicity free as
a representation of $\sl(2)$.
\end{theorem}

\begin{remark}
Items (i), (ii) and (iii) are not mutually exclusive since they have intersection when 
$\ell=0$ or $\ell'=0$. 
We can make them exclusive by requiring $\ell,\ell'\ge 1$ in (ii).
In that case, the second sum in $\text{soc}(V \otimes W)$ in (ii)
is non-empty if and only if $a_0+m\le b_0$.
\end{remark}

\begin{remark}\label{rmk:A2Am}
If $U\simeq V\otimes W$ with $V$ and $W$ as in Theorem \ref{thm:main}, then the list of 
highest weights appearing in 
the $\sl(2)$-decomposition of $\text{soc}(U)$
consists of the union of two sets $A_2(U)$ and $A_m(U)$ whose elements are in 
arithmetic progressions with common differences 
2 and $m$ respectively.  
In all cases 
\[
A_2(U)=\{|a_0-b_0|,|a_0-b_0|+2,\dots,a_0+b_0\}
\]
($A_2(U)$ consists of all the highest weights of 
$\text{soc}(V)\otimes\text{soc}(W)$)
and 
\[
A_m(U)=
\begin{cases}
\left\{a_0+b_0+m,\dots,a_0+b_0+\min\{\ell,\ell'\}m\right\}, &
\text{in case (i) with $\ell,\ell'\ge 1$}; \\[2mm]
\left\{b_0\!-\!a_0\!-\!\min\{\ell,\ell',\lfloor\frac{b_0-a_0}{m}\rfloor\}m,\dots,b_0\!-\!a_0\!-\!m\right\}, &
\text{in case (ii) with $\ell,\ell'\ge 1$} \\[-1mm]
&
\text{and $b_0\ge a_0+m$}; \\[2mm]
\emptyset, &
\text{ otherwise.}
\end{cases}
\]
Note that $A_2(U)$ and $A_m(U)$ are disjoint sets and 
given $A_2(U)\cup A_m(U)$, 
if $m\ne 2$, it is clear how to identify $A_2(U)$ and $A_m(U)$.
\end{remark}

\begin{proof}[Proof of Theorem \ref{thm:main}] 
It follows from \eqref{eq:soc=socxsoc+S_0}
and the Clebsch-Gordan formula 
for the decomposition of the tensor product
of irreducible $\sl(2)$-modules, that 
in order to prove this theorem we only need to study 
$S_t$. 

Let us fix $t_0>1$ and assume 
$S_{t_0}\ne 0$.
We know from Proposition 
\ref{prop:Zocalo_t} that $t_0\le\min\{ \ell,\ell'\}$.
We will first show that if $\mu$ is the weight of 
a highest weight 
vector
\[
u\in S_{t_0}=S_{t_0}(V,W)=\Big(\bigoplus_{i+j=t_0} V(a_{i})\otimes V(b_j)\Big)^{\mathfrak{r}},
\]
then $\mu$ is as claimed in the theorem. 

We have  
\[
u=\bigoplus_{i+j=t_0} u_{i,j}
\]
with $u_{i,j}\in V(a_{i})\otimes V(b_j)$, then each
$u_{i,j}$ must be a highest weight 
vector of weight $\mu$.
Moreover, since $\mathfrak{r}u=0$ it follows 
from \eqref{eq:r in type Z} that 
\begin{equation}\label{eq:r suma 0 0}
(X\,u_{i,j})_2+(X\,u_{i+1,j-1})_1=0
\end{equation}
for all $X\in\mathfrak{r}$ and for all $i,j$ such that $i+j=t_0$ and $i+1\le\ell$ and $j\ge1$. 
In other words,
let us fix $i_0,j_0$ so that 
$i_0+j_0=t_0-1$. In particular  $i_0<\ell$ and $j_0<\ell'$. 
Now we may consider 
uniserial subquotients $\tilde V$
and $\tilde W$ 
whose socle decompositions are
\[
\tilde V=V(a_{i_0})\oplus V(a_{i_0+1})\qquad 
\tilde W=V(b_{j_0})\oplus V(b_{j_0+1}).
\]
Then, \eqref{eq:r suma 0 0} is equivalent to say that 
\[
u_{i_0,j_0+1}+u_{i_0+1,j_0}\in S_1(\tilde V,\tilde W)
\]
and we know that it is a highest weight 
vector of weight $\mu$.

We now apply Theorem \ref{thm:length_2} and we obtain
that 
\begin{enumerate}[(i)]
\item If $V=Z(a_0,\ell)$ and 
$W=Z(b_0,\ell')$, then $a_i=a_0+im$ and $b_j=b_0+jm$.
This implies
$V=Z(a_{i_0},1)$, 
$W=Z(b_{j_0},1)$ and hence 
$S_1(\tilde V,\tilde W)\simeq V(a_{i_0}+b_{j_0}+m)$.
that is 
\[
\mu=a_{i_0}+b_{j_0}+m=a_{0}+i_0m+b_{0}+j_0m+m= a_{0}+b_{0}+t_0m.
\]
\item If $V=Z(a_0,\ell)$ and 
$W=Z(b_0,\ell')^*$, then 
$a_i=a_0+im$ and $b_j=b_0-jm$.
This implies
$V=Z(a_{i_0},1)$,  
$W=Z(b_{j_0+1},1)^*$ and hence 
\[
S_1(\tilde V,\tilde W)\simeq\begin{cases} 
V(b_{j_0+1}-a_{i_0}),& 
\text{if $a_{i_0}\le b_{j_0+1}$;} \\ 
0,& \text{if $a_{i_0}> b_{j_0+1}$;} \end{cases}
\]
which is equivalent to 
\[
S_1(\tilde V,\tilde W)\simeq\begin{cases} 
V(b_{0}-a_{0}-t_0m),& 
\text{if $a_{0}+i_0m \le b_{0}-j_0m-m$;} \\ 
0,& \text{if $a_{0}+i_0m > b_{0}-j_0m-m$.} \end{cases}
\]
Thus, if $S_1(\tilde V,\tilde W)\ne 0$ then 
$\mu=b_{0}-a_{0}-t_0m$.
\item If $V=Z(a_0,\ell)^*$ and 
$W=Z(b_0,\ell')^*$, then $S_1(\tilde V,\tilde W)= 0$.
\end{enumerate}

This completes the first part of the proof, that is 
if $S_{t_0}(V,W)\ne 0$
then the have proved that only highest weight $\mu$ 
appearing in $S_{t_0}(V,W)$ are as claimed. 

Conversely, assume that $\mu$ is a weight 
claimed to appear in $S_{t_0}(V,W)$. 
The above analysis shows, in each case, 
that $S_1(\tilde V,\tilde W)$ is isomorphic to 
$V(\mu)$ for all $i_0+j_0=t_0-1$.
Now we can choose a highest weight vector 
\[
u_{i_0,j_0+1}+u_{i_0+1,j_0}\in 
V(a_{i_0})\otimes V(a_{j_0+1})
\; \oplus  \;
V(a_{i_0+1})\otimes V(b_{j_0})\subset
S_1(\tilde V,\tilde W)
\]
in a recursive way so that 
\[
\sum_{i=0}^{t_0-1}
u_{i,t_0-i}+u_{i+1,t_0-i-1}
\]
is a highest weight vector of weight $\mu$ in 
$S_{t_0}(V,W)$.
This completes the proof of the theorem.
\end{proof}

\begin{corollary}\label{coro:soc_completo}
Let 
$V=V(a_0)\oplus \hdots \oplus V(a_\ell)$ and 
$W=V(b_0)\oplus \hdots \oplus V(b_{\ell'})$
be socle decomposition of two uniserial $\g_m$-modules of type $Z$.
Then:
\begin{enumerate}[(i)]
\item
$\text{soc}(V \otimes W)=\text{soc}(V)\otimes \text{soc}(W)$
if and only if 
\begin{enumerate}[\hspace{1mm}(a)]
\item $\ell=0$ or $\ell'=0$,
\item $V=Z(a_{\ell},\ell)^*$ and 
$W=Z(b_{\ell'},\ell')^*$, $\ell,\ell'\ge 0$,
\item $V=Z(a_0,\ell)$ and 
$W=Z(b_{\ell'},\ell')^*$ with $b_0<a_0+m$ and $\ell,\ell'\ge 1$,
\item  $V=Z(a_{\ell},\ell)^*$ and 
$W=Z(b_{0},\ell')$ with $a_0<b_0+m$ and $\ell,\ell'\ge 1$.
\end{enumerate}

\medskip

\item In any of the cases described in (i),
the socle length of $V\otimes W$ is 
$\ell+\ell'+1$ and 
\[
\text{soc}^{t+1}(V\otimes W)=\sum_{i=0}^{t}
\text{soc}^{i+1}(V)\otimes 
\text{soc}^{t+1-i}(W)=\bigoplus_{0\le i+j\le t}
V(a_i)\otimes V(b_j)
\]
as $\sl(2)$-modules 
for all $0\le t\le \ell+\ell'$.
\end{enumerate}
\end{corollary}

\begin{proof}
Part (i) follows at once from Theorem \ref{thm:main}.
Part (ii) is a consequence of Lemma \ref{lemma:soc_series}
applied to the decomposition 
\[
V\otimes W=\bigoplus_{k=1}^{\ell+\ell'+1}(V\otimes W)_k
\]
with $(V\otimes W)_k=\bigoplus_{i+j=k-1} V(a_i)\otimes V(b_j)$.
The hypothesis $\mathfrak{r} V_k \subset V_{k-1}$ required by the 
lemma follows from \eqref{eq:r in type Z} 
and since we are in the cases described in (i) we 
have $\text{soc}(V\otimes W)=(V\otimes W)_1$
as required by the lemma.
\end{proof}
 
 \section{Applications} 
 Recall from \S\ref{wdos} that 
a uniserial $\g_m$-module is of type $Z$ if it is isomorphic to
$Z(a,\ell)$ or $Z(a,\ell)^*$ for some non-negative integers 
$a$ and $\ell$.

\subsection{Invariants and intertwining operators}\label{sec:intertwining}

The main goal of this subsection is to obtain 
the intertwining operators between
two uniserial $\g_m$-modules $V$ and $W$ of type $Z$.

Since $\text{Hom}(V,W)\simeq V^*\otimes W$ as $\g_m$-modules, we have
\[
\text{Hom}_{\g_m}(V,W)\simeq (V^*\otimes W)^{\g_m}.
\]
In turn, since $\text{soc}(V^*\otimes W)=(V^*\otimes W)^{\mathfrak{r}}$
 (see Lemma \ref{lemma:soc} and \eqref{eq:soc=r-invariants}),
it follows that 
$(V^*\otimes W)^{\g_m}$
is the subspace of $\sl(2)$-invariant vectors in 
$\text{soc}(V^*\otimes W)$.
Since Theorem \ref{thm:main} shows that $\sl(2)$-decomposition 
is multiplicity free, it follows immediately that 
$\dim\text{Hom}_{\g_m}(V,W)$ (or $\dim(V^*\otimes W)^{\g_m}$) is either 0 or 1. 

The following two corollaries describe exactly in which cases these dimensions are 1. 

 \begin{corollary}\label{thm:invariants}
Let 
$V=V(a_0)\oplus \hdots \oplus V(a_\ell)$ and 
$W=V(b_0)\oplus \hdots \oplus V(b_{\ell'})$
be socle decomposition of two uniserial $\g_m$-modules of type $Z$.
Then
\[
(V\otimes W)^{\g_m}\ne 0
\]
if and only if 
$b_0\in \{a_0, \hdots, a_\ell\}$ and 
$a_0\in \{b_0, \hdots, b_{\ell'}\}$ and in this case $(V\otimes W)^{\g_m}$
is 1-dimensional. 
\end{corollary}
\begin{proof}
We have to consider the possibilities $V$ and $W$ isomorphic to 
$Z(c,k)$ or $Z(c,k)^*$. 

If  either 
$V=Z(a_0,\ell)$ and 
$W=Z(b_0,\ell')$, or 
$V=Z(a_{\ell},\ell)^*$ and 
$W=Z(b_{\ell},\ell')^*$,   
then Theorem \ref{thm:main} implies 
that the trivial $\sl(2)$-module appears in
$\text{soc}(V \otimes W)$ if and only if 
$b_0=a_0$. 
In the first and second cases the sequences 
$\{a_i\}$ and $\{b_j\}$ look like:
\begin{equation*}
\begin{matrix}
a_0 & a_0+m & a_0 +2m & \hdots  \\[2mm]
 \verteq  &  \verteq  &  \verteq & \\
b_0 &  b_0+m & b_0 +2m & \hdots .
\end{matrix}
\qquad\text{ or }\qquad
\begin{matrix}
 \hdots & a_0-2m & a_0-m & a_0 \\[2mm]
 &\verteq  &  \verteq  &  \verteq  \\
 \hdots & b_0-2m & b_0-m & b_0
\end{matrix}
\end{equation*}
respectively. 
In these cases, this precisely coincides with the condition 
$b_0\in \{a_i:i=0,\dots\ell\}$ and 
$a_0\in \{b_j:j=0,\dots\ell'\}$.

If   
$V=Z(a_0,\ell)$ and 
$W=Z(b_{\ell'},\ell')^*$, it follows from 
Theorem \ref{thm:main} 
that $\text{soc}(V \otimes W)$  contains the trivial representation of $\sl(2)$ if and only if 
$b_0=a_0+tm$ with $0\le t\le\min\{\ell,\ell'\}$.
Thus, the sequences 
$\{a_i\}$ and $\{b_j\}$ look like:
\begin{equation*}
\begin{matrix}
 &&a_0 &  \hdots & a_0 +tm & \hdots &a_0+\ell m \\[2mm]
 &&\verteq  &&  \verteq  &         &     &        & \\
 b_0-{\ell'}m &\hdots &b_0-tm & \hdots&  b_0.
\end{matrix}
\end{equation*}
Again, in this case, this precisely coincides with the condition 
$b_0\in \{a_i:i=0,\dots\ell\}$ and 
$a_0\in \{b_j:j=0,\dots\ell'\}$.

The   case 
$V=Z(a_{\ell},\ell)^*$ and 
$W=Z(b_{0},\ell')$ is symmetric to the previous one.
\end{proof}

 \begin{corollary}\label{thm:intertwining}
Let 
$V=V(a_0)\oplus \hdots \oplus V(a_\ell)$ and 
$W=V(b_0)\oplus \hdots \oplus V(b_{\ell'})$
be socle decomposition of two uniserial $\g_m$-modules of type $Z$.
Then $\dim\text{Hom}_{\g_m}(V,W)$ is either 0 or 1 and  
\[
\dim\text{Hom}_{\g_m}(V,W)=1
\]
if and only if 
$b_0\in \{a_0, \hdots, a_\ell\}$ and 
$a_\ell\in \{b_0, \hdots, b_{\ell'}\}$, that is, these sets look like
\begin{equation*}
\begin{matrix}
a_0 & \hdots & a_i & \hdots  & a_\ell &        & \\[2mm]
           &&  \verteq  &         & \verteq     &        & \\
           &&  b_0 & \hdots & b_j     & \hdots & b_{\ell'}.
\end{matrix}
\end{equation*}
\end{corollary}
\begin{proof}
This is basically a direct consequence of Corollary \ref{thm:invariants}.
Since 
\[
\text{Hom}_{\g_m}(V,W)\simeq (V^*\otimes W)^{\g_m}
\]
we need to apply Corollary  \ref{thm:invariants} to the 
$\g_m$-modules $V^*$ and $W$ whose socle decomposition are
$V^*=V(a_\ell)\oplus \hdots \oplus V(a_0)$ and 
$W=V(b_0)\oplus \hdots \oplus V(b_{\ell'})$ respectively.
Therefore 
$(V^*\otimes W)^{\g_m}\ne0$ if and only if
$b_0\in \{a_0, \hdots, a_\ell\}$ and 
$a_\ell\in \{b_0, \hdots, b_{\ell'}\}$.
\end{proof}

\subsection{Isomorphisms between tensor products}\label{sec:isomorphisms}
In this section we use Theorem \ref{thm:main} to prove, for $m\ne 2$,  
that if $U$ is the tensor product of 
two uniserial $\g_m$-modules of type $Z$, then the factors are determined by $U$.

If $U$ is the tensor product of 
two uniserial $\g_m$-modules of type $Z$, then so is 
$U^*$. Recall also that, in this case,
Theorem \ref{thm:main} implies that 
the list of 
highest weights appearing in 
the $\sl(2)$-decomposition of $\text{soc}(U)$
consists of the union of two disjoint 
sets $A_2(U)$ and $A_m(U)$ whose elements are in 
arithmetic progressions with common differences 
2 and $m$ respectively (see Remark \ref{rmk:A2Am}).  
Thus we have
\[
\text{soc}(U)\simeq
\bigoplus_{k\in A_2(U)} V(k)
\;\oplus\;
\bigoplus_{k\in A_m(U)} V(k).
\]
We know that always $A_2(U)\ne\emptyset$ but $A_m(U)$ might be empty
and, as pointed out in  Remark \ref{rmk:A2Am}, both sets can be obtained from $U$ when $m\ne 2$. If $m=2$, we do not know whether it is possible 
to read off $A_2(U)$ and $A_m(U)$ from $U$.

 \begin{theorem}\label{thm:isomorphism}
 Let $m\ne 2$ and let $U$ be the tensor product of 
two uniserial $\g_m$-modules of type $Z$.
Then the factors are determined by $U$.
More precisely, let $\Lambda$ be 
the greatest (highest) weight in $U$ and let
$A_2=A_2(U)$, $A_m=A_m(U)$, 
$A_2^*=A_2(U^*)$ and $A_m^*=A_m(U^*)$.
Then:
\begin{enumerate}[(i)]
\item Assume that 
$\max A_2^*=\Lambda$.
 Set
\[
\ell'=\begin{cases}
\dfrac{\max A_m-\max A_2}{m}, & 
\text{if $A_m\ne\emptyset$;} \\
0, & \text{if $A_m=\emptyset$;} 
\end{cases}\qquad 
\ell=\frac{\max A_2^*-\max A_2}{m}-\ell'.
\]
If 
$(\ell-\ell')m=\min A_2^* - \min A_2$ 
then set 
\[
a=\frac{\max A_2+\min A_2}{2},\quad 
b=\frac{\max A_2-\min A_2}{2},
\]
else set 
\[
a=\frac{\max A_2-\min A_2}{2},\quad 
b=\frac{\max A_2+\min A_2}{2}.
\]
We have $U\simeq Z(a,\ell)\otimes Z(b,\ell')$.

\medskip

\item Assume 
$\max A_2=\Lambda$.
Set
\[
\ell'=\begin{cases}
\dfrac{\max A_m^*-\max A_2^*}{m}, & 
\text{if $A_m^*\ne\emptyset$;} \\
0, & \text{if $A_m^*=\emptyset$;} 
\end{cases}\qquad 
\ell=\frac{\max A_2-\max A_2^*}{m}-\ell'.
\]
If 
$(\ell-\ell')m=\min A_2 - \min A_2^*$ 
then set 
\[
a=\frac{\max A_2^*+\min A_2^*}{2},\quad 
b=\frac{\max A_2^*-\min A_2^*}{2},
\]
else set 
\[
a=\frac{\max A_2^*-\min A_2^*}{2},\quad 
b=\frac{\max A_2^*+\min A_2^*}{2}.
\]
We have $U\simeq Z(a,\ell)^*\otimes Z(b,\ell')^*$.

\medskip

\item Assume that neither $\max A_2$ nor 
$\max A_2^*$ is $\Lambda$. 
Set
\[
\ell'=
\frac{\Lambda-\max A_2^*}{m},
\qquad 
\ell=\frac{\Lambda-\max A_2}{m}.
\]
If 
$(\ell+\ell')m=\min A_2^* - \min A_2$ 
then set 
\[
a=\frac{\max A_2+\min A_2}{2},\quad 
b=\frac{\max A_2-\min A_2}{2}-\ell'm,
\]
else set 
\[
a=\frac{\max A_2-\min A_2}{2},\quad 
b=\frac{\max A_2+\min A_2}{2}-\ell'm.
\]
We have $U\simeq Z(a,\ell)\otimes Z(b,\ell')^*$.
\end{enumerate}  
\end{theorem}

\begin{proof}
We know that $U$ is one of the following possibilities:
\[
Z(a_0,\ell_0)\otimes Z(b_0,\ell'_0),\quad
Z(a_0,\ell_0)^*\otimes Z(b_0,\ell'_0)^*,\quad
Z(a_0,\ell_0)\otimes Z(b_0,\ell'_0)^*.
\]
In any case, $\Lambda=a_0+b_0+(\ell_0+\ell'_0)m$.

In order to apply Theorem \ref{thm:main}, it is convenient 
to recall that the socle decompositions of the modules 
$Z(c,t)$ and $Z(c,t)^*$ are
\begin{align*}
Z(c,t)&=V(c)\oplus V(c+m)\oplus
\dots\oplus V(c+t m), \\
Z(c,t)^*&=V(c+t m)\oplus V(c+(t-1)m)\oplus
\dots\oplus V(c).
\end{align*}
Thus (see Remark \ref{rmk:A2Am})

\noindent
If $U=Z(a_0,\ell_0)\otimes Z(b_0,\ell'_0)$ then 
\begin{align*}
A_2&=\{|a_0-b_0|,|a_0-b_0|+2,\dots,a_0+b_0\}, \\[2mm]
A_2^*&=\{|a_0\!+\!\ell_0 m-b_0\!-\!\ell'_0 m|,
|a_0\!+\!\ell_0 m-b_0\!-\!\ell'_0 m|+2,\dots,a_0\!+\!\ell_0 m+b_0\!+\!\ell'_0 m\}, \\[2mm]
A_m&=\{a_0+b_0+m,\dots,a_0+b_0+\min\{\ell_0,\ell'_0\}m\}
\text{ if $\ell_0,\ell'_0>0$, else $A_m=\emptyset$}, \\[2mm]
A_m^*&=\emptyset.
\end{align*}

\noindent
If $U=Z(a_0,\ell_0)^*\otimes Z(b_0,\ell'_0)^*$ then 
\begin{align*}
A_2^*&=\{|a_0-b_0|,|a_0-b_0|+2,\dots,a_0+b_0\}, \\[2mm]
A_2&=\{|a_0\!+\!\ell_0 m-b_0\!-\!\ell'_0 m|,
|a_0\!+\!\ell_0 m-b_0\!-\!\ell'_0 m|+2,\dots,a_0\!+\!\ell_0 m+b_0\!+\!\ell'_0 m\}, \\[2mm]
A_m^*&=\{a_0+b_0+m,\dots,a_0+b_0+\min\{\ell_0,\ell'_0\}m\}
\text{ if $\ell_0,\ell'_0>0$, else $A_m=\emptyset$}, \\[2mm]
A_m&=\emptyset.
\end{align*}

\noindent
If $U=Z(a_0,\ell_0)\otimes Z(b_0,\ell'_0)^*$
and $\ell_0,\ell'_0>0$
then 
\begin{align*}
A_2&=\{|a_0-b_0-\ell'_0 m|,|a_0-b_0-\ell'_0 m|+2,\dots,a_0+b_0+\ell'_0 m\}, \\[2mm]
A_2^*&=\{|a_0+\ell_0 m-b_0|,
|a_0+\ell_0 m-b_0|+2,\dots,a_0+\ell_0 m+b_0\}, \\[2mm]
A_m&=\left\{b_0+\ell'_0 m-a_0-tm:t=1,\dots,T
\right\}
 \text{ if $b_0+\ell'_0 m-a_0\ge m$, else $A_m=\emptyset$}, \\[2mm]
A_m^*&=\left\{a_0+\ell_0 m-b_0-tm:t=1,\dots,T^* \right\}
 \text{ if $a_0+\ell_0 m-b_0\ge m$, else $A_m=\emptyset$},
\end{align*}
with 
\[
T=\min\left\{\ell_0,\ell'_0,
\left\lfloor\tfrac{b_0+\ell'_0 m-a_0}{m}\right\rfloor\right\},\qquad
T^*=\min\left\{\ell_0,\ell'_0,
\left\lfloor\tfrac{a_0+\ell_0 m-b_0}{m}\right\rfloor\right\}.
\]
In what follows we need the following fact: given $x,y\ge0$
and $z>0$ then 
\begin{equation}\label{eq:xyz}
|x-y+z|-|x-y|=z\quad\text{if and only if}\quad x\ge y.
\end{equation}

Now we begin the proof.
Suppose first that $\ell_0=\ell'_0=0$.
Then $U\simeq Z(a_0,0)\otimes Z(b_0,0)$ and 
$U$ falls into cases (i) and (ii). 
Either (i) or (ii) 
yield $\ell=\ell'=0$ and $a=\max\{a_0,b_0\}$, 
$b=\min\{a_0,b_0\}$, which is correct.

Suppose now that $\ell'_0=0$ and $\ell_0>0$.
Then either 
$U\simeq Z(a_0,\ell)\otimes Z(b_0,0)$ or 
$U\simeq Z(a_0,\ell)^*\otimes Z(b_0,0)$. 

If 
$U\simeq Z(a_0,\ell)\otimes Z(b_0,0)$ 
then $U$ falls only in case (i) (since 
$\max A_2<\Lambda$) and this 
yield $\ell'=0$, $\ell=\ell_0$.
Since 
\[
\min A_2=|a_0-b_0|\quad\text{ and }\quad 
\min A_2^*=|a_0+\ell m-b_0|
\]
it follows that $\min A_2^*-\min A_2=\ell m$ if and only if 
$a_0\ge b_0$. 
In any case, we obtain $a=a_0$ and $b=b_0$.

Similarly,  
$U\simeq Z(a_0,\ell)^*\otimes Z(b_0,0)$ 
then $U$ falls only in case (ii) (since 
$\max A_2^*<\Lambda$) and this 
yield $\ell'=0$, $\ell=\ell_0$.
Now 
\[
\min A_2=|a_0+\ell m-b_0|\quad\text{ and }\quad 
\min A_2^*=|a_0-b_0|
\]
and hence $\min A_2-\min A_2^*=\ell m$ if and only if 
$a_0\ge b_0$. 
In any case, we obtain $a=a_0$ and $b=b_0$.

We finally suppose that $\ell_0,\ell'_0>0$.
Then cases (i), (ii), (iii) correspond
exactly (without superposition) to the cases when $U$ is isomorphic to 
$Z(a_0,\ell_0)\otimes Z(b_0,\ell'_0)$, or 
$Z(a_0,\ell_0)^*\otimes Z(b_0,\ell'_0)^*$, or
$Z(a_0,\ell_0)\otimes Z(b_0,\ell'_0)^*$ respectively.

In cases (i) and (ii) we obtain 
$\ell=\max\{\ell_0,\ell'_0\}$, 
$\ell'=\min\{\ell_0,\ell'_0\}$.
We need to show that $a_0$ and $b_0$ are obtained correctly. 

In case (i) we have to distinguish four cases:
\[
\begin{array}{l}
\ell_0\ge\ell'_0  \\
 a_0\ge b_0
\end{array},\qquad
\begin{array}{l}
\ell_0\ge\ell'_0  \\
 a_0\le b_0
\end{array},\qquad\begin{array}{l}
\ell_0\le\ell'_0  \\
 a_0\ge b_0
\end{array},\qquad\begin{array}{l}
\ell_0\le\ell'_0  \\
 a_0\le b_0
\end{array}.
\]
Let us consider, for instance, the third case. 
We obtain $\ell=\ell'_0$ and $\ell'=\ell_0$. 
If $a_0=b_0$ it is clear that the result will be correct.
Otherwise, $b_0< a_0$ and since 
\[
\min A_2=|b_0-a_0|\quad\text{ and }\quad 
\min A_2^*=|b_0-a_0+(\ell'_0-\ell_0) m|
\]
it follows from \eqref{eq:xyz} that 
$(\ell-\ell')m\ne \min A_2^* - \min A_2$ 
and hence $a=b_0$ and $b=b_0$, which is correct. 
All the other three cases work similarly.

The case (ii) is analogous to the case (i).

In case (iii) we obtain $\ell=\ell_0$ and $\ell'=\ell'_0$.
Now, since 
\[
\min A_2=|a_0-b_0-\ell' m|\quad\text{ and }\quad 
\min A_2^*=|a_0-b_0+\ell m|=|a_0-b_0-\ell' m+(\ell+\ell') m|
\]
it follows from \eqref{eq:xyz} that 
$(\ell+\ell')m=\min A_2^* - \min A_2$ 
if and only if $a_0\ge b_0+\ell' m$.
In either case we obtain $a=a_0$ and $b=b_0$, which is correct. 

This completes the proof.
\end{proof}

\bibliographystyle{plain}
\bibliography{bibliografia}

\begin{thebibliography}{10}

\bibitem{ASS}
I.~Assem, D.~Simson, and A.~Skowro\'nski.
\newblock {\em Elements of the Representation Theory of Associative Algebras:
  1. Techniques of the Representation Theory}.
\newblock Cambridge University Press, New York (USA), UK, 2007.

\bibitem{ARS}
M.~Auslander, I.~Reiten, and S.O. Smal\o.
\newblock {\em Representation Theory of Artin Algebras}.
\newblock Cambridge University Press, New York (USA), Melbourne (Australia),
  1995.

\bibitem{BGG}
I.~N. Bernstein, I.M. Gelfand, and S.I. Gelfand.
\newblock {Structure of Representations that are generated by vectors of
  highest weight}.
\newblock {\em Functional. Anal. Appl.}, 5:1--8, 1971.

\bibitem{BK}
Christine Bessenrodt and Alexander Kleshchev.
\newblock Irreducible tensor products over alternating groups.
\newblock {\em Journal of Algebra}, 228:536--550, 06 2000.

\bibitem{BH-Z}
K.~Bongartz and B.~Huisgen-Zimmermann.
\newblock {The geometry of uniserial representations of algebras II. Alternate
  viewpoints and uniqueness}.
\newblock {\em J. Pure Appl. Algebra}, 157:23--32, 2001.

\bibitem{Bo}
N.~Bourbaki.
\newblock {\em Alg{\`e}bres de Lie}.
\newblock Paris: Hermann, 1960.

\bibitem{CGS2}
L.~Cagliero, L.~Guti\'errez Frez, and F.~Szechtman.
\newblock {Classification of finite dimensional uniserial representations of
  conformal Galilei algebras}.
\newblock {\em Journal of Mathematical Physics}, 57(101706), 2016.

\bibitem{CGS1}
L.~Cagliero, L.~Guti\'errez Frez, and F.~Szechtman.
\newblock {Free 2-step nilpotent Lie algebras and indecomposable modules}.
\newblock {\em Comm. Algebra}, 46:2990--3005, 2018.

\bibitem{CLS}
L.~Cagliero, F.~Levstein, and F.~Szechtman.
\newblock {Nilpotency degree of the nilradical of a solvable Lie algebra on two
  generators and uniserial modules associated to free nilpotent Lie algebras}.
\newblock {\em Journal of Algebra}, 585:447--483, 2021.

\bibitem{CS_JofAlg}
L.~Cagliero and F.~Szechtman.
\newblock {The classification of uniserial $\sl(2)\ltimes V(m)$-modules and a
  new interpretation of the Racah-Wigner $6j$-symbol}.
\newblock {\em Journal Algebra}, 386:142--175, 2013.

\bibitem{CS_canadian}
L.~Cagliero and F.~Szechtman.
\newblock {On the theorem of the primitive element with applications to the
  representation theory of associative and Lie algebras}.
\newblock {\em Canad. Math. Bull.}, 57:735--748, 2014.

\bibitem{CS_JofAlgApp}
L.~Cagliero and F.~Szechtman.
\newblock {Classification of linked indecomposable modules of a family of
  solvable Lie algebras over an arbitrary field of characteristic 0}.
\newblock {\em J. of Algebra and Its Applications}, 15(1650029), 2016.

\bibitem{CS_Comm}
L.~Cagliero and F.~Szechtman.
\newblock {Indecomposable modules of 2-step solvable Lie algebras in arbitrary
  characteristic}.
\newblock {\em Comm. Algebra}, 44:1--10, 2016.

\bibitem{Ca2}
P.~Casati.
\newblock {Irreducible $\sl_{n+1}$-representations remain indecomposable
  restricted to some Abelian subalgebras}.
\newblock {\em Journal Lie Theory}, 20:393--407, 2010.

\bibitem{Ca1}
P.~Casati.
\newblock {The classification of the perfect cyclic $\mathfrak{sl}_{n+1}\ltimes
  \mathbb{C}^{n+1}$-modules}.
\newblock {\em Journal of Algebra}, 476:311--343, 2017.

\bibitem{CMS}
P.~Casati, S.~Minniti, and V.~Salari.
\newblock {Indecomposable representations of the Diamond Lie algebra}.
\newblock {\em Journal of Mathematical Physics}, 51(033515):20pp, 2010.

\bibitem{Casati2017IndecomposableMO}
Paolo Casati, Andrea Previtali, and Fernando Szechtman.
\newblock Indecomposable modules of a family of solvable lie algebras.
\newblock {\em Linear Algebra and its Applications}, 531:423--446, 2017.

\bibitem{CM}
V.~Chari and A.~Moura.
\newblock {The restricted Kirillov-Reshetikhin modules for the current and
  twisted current algebras}.
\newblock {\em Commun. Math. Phys.}, 266:431--454, 2006.

\bibitem{Dd}
A.~Douglas and H.~de~Guise.
\newblock {Some nonunitary, indecomposable representations of the Euclidean
  algebra $\mathfrak{e}(3)$}.
\newblock {\em J. Phys. A: Math. Theor.}, 43(085204):13pp, 2010.

\bibitem{DKR}
A.~Douglas, D.~Kahrobaei, and J.~Repka.
\newblock {Classification of embeddings of abelian extensions of $D_n$ into
  $E_{n+1}$}.
\newblock {\em J. Pure Appl. Algebra}, 217:1942--1954, 2013.

\bibitem{DP}
A.~Douglas and A.~Premat.
\newblock {A class of nonunitary, finite dimensional representations of the
  euclidean algebra $\mathfrak{e}(2)$}.
\newblock {\em Communications in Algebra}, 35:1433--1448, 2007.

\bibitem{GP}
I.M. Gelfand and V.A. Ponomarev.
\newblock {Remarks on the classification of a pair of commuting linear
  transformations in a finite dimensional vector space}.
\newblock {\em Functional Anal. Appl.}, 3:325--326, 1969.

\bibitem{H-Z}
B.~Huisgen-Zimmermann.
\newblock {The geometry of uniserial representations of finite dimensional
  algebras. III: Finite uniserial type}.
\newblock {\em Trans. Amer. Math. Soc.}, 348:4775--4812, 1996.

\bibitem{J}
H.~P. Jakobsen.
\newblock {\em Indecomposable finite-dimensional representations of a class of
  Lie algebras and Lie superalgebras}, volume 2027.
\newblock Lecture Notes in Math., Springer, Heidelberg, 2011.

\bibitem{MOROTTI_JPAA}
Lucia Morotti.
\newblock Irreducible tensor products for alternating groups in characteristics
  2 and 3.
\newblock {\em J. of Pure and Applied Algebra}, 224(12):106426, 2020.

\bibitem{MOROTTI_Rep_Theory}
Lucia Morotti.
\newblock Irreducible tensor products of representations of covering groups of
  symmetric and alternating groups.
\newblock {\em Represent. Theory}, 25:543--593, 2021.

\bibitem{M}
A.~Moura.
\newblock {Restricted limits of minimal affinizations}.
\newblock {\em Pacific J. Math}, 244:359--397, 2010.

\bibitem{Na}
T.~Nakayama.
\newblock {On Frobeniusean algebras II}.
\newblock {\em Ann. of Math.}, 42:1--21, 1941.

\bibitem{NGB}
Z.~Nazemian, A.~Ghorbani, and M.~Behboodi.
\newblock {Uniserial dimension of modules}.
\newblock {\em J. Algebra}, 399:894--903, 2014.

\bibitem{Pi1}
A.~Piard.
\newblock {Sur des repr\'esentations ind\'ecomposables de dimension finie de
  $SL(2).R^2$}.
\newblock {\em Journal of Geometry and Physics}, 3:1--53, 1986.

\bibitem{Pu}
G.~Puninski.
\newblock {\em Serial rings}.
\newblock Kluwer Academic Publishers, Dordrecht, 2001.

\bibitem{Ra}
C.~S. Rajan.
\newblock {Unique decomposition of tensor products of irreducible
  representations of simple algebraic groups}.
\newblock {\em Ann. of Math.}, 160:683--704, 2004.

\bibitem{Reif2021OnTP}
Shifra Reif and R.~Venkatesh.
\newblock On tensor products of irreducible integrable representations.
\newblock {\em Journal of Algebra}, 2021.

\bibitem{Sa}
A.~Savage.
\newblock {Quivers and the Euclidean group}.
\newblock {\em Contemporary Mathematics}, 478:177--188, 2009.

\bibitem{VMK}
D.~A. Varshalovich, A.~N. Moskalev, and V.~K. Khersonskii.
\newblock {\em Quantum theory of angular momentum}.
\newblock World Scientific, Singapore, 1989.

\bibitem{Venkatesh2012UniqueFO}
R.~Venkatesh and Sankaran Viswanath.
\newblock Unique factorization of tensor products for kac-moody algebras.
\newblock {\em arXiv: Representation Theory}, 2012.

\end{thebibliography}

\end{document}